\documentclass[11pt]{amsart}
\usepackage{amsopn}
\usepackage{amssymb, amscd}
\usepackage{multirow}
\usepackage{graphicx, graphics, epsfig}
\usepackage{faktor}
\usepackage{enumerate}
\usepackage{tikz}
\usepackage{pgfplots}
\usepackage{hyperref}
\usepackage{color}

\newcommand{\nc}{\newcommand}

\nc{\fg}{\mathfrak{f} } \nc{\vg}{\mathfrak{v} } \nc{\wg}{\mathfrak{w} }
\nc{\zg}{\mathfrak{z} } \nc{\ngo}{\mathfrak{n} } \nc{\kg}{\mathfrak{k} }
\nc{\mg}{\mathfrak{m} } \nc{\bg}{\mathfrak{b} } \nc{\ggo}{\mathfrak{g} } \nc{\eg}{\mathfrak{e} }
\nc{\ggob}{\overline{\mathfrak{g}} } \nc{\sog}{\mathfrak{so} }
\nc{\sug}{\mathfrak{su} } \nc{\spg}{\mathfrak{sp} } \nc{\slg}{\mathfrak{sl} }
\nc{\glg}{\mathfrak{gl} } \nc{\cg}{\mathfrak{c} } \nc{\rg}{\mathfrak{r} }
\nc{\hg}{\mathfrak{h} } \nc{\tg}{\mathfrak{t} } \nc{\ug}{\mathfrak{u} }
\nc{\dg}{\mathfrak{d} } \nc{\ag}{\mathfrak{a} } \nc{\pg}{\mathfrak{p} }
\nc{\sg}{\mathfrak{s} } \nc{\affg}{\mathfrak{aff} } \nc{\qg}{\mathfrak{q} } \nc{\lgo}{\mathfrak{l} }

\nc{\pca}{\mathcal{P}} \nc{\nca}{\mathcal{N}} \nc{\lca}{\mathcal{L}}
\nc{\oca}{\mathcal{O}} \nc{\mca}{\mathcal{M}} \nc{\tca}{\mathcal{T}}
\nc{\aca}{\mathcal{A}} \nc{\cca}{\mathcal{C}} \nc{\gca}{\mathcal{G}}
\nc{\sca}{\mathcal{S}} \nc{\hca}{\mathcal{H}} \nc{\bca}{\mathcal{B}}
\nc{\dca}{\mathcal{D}} \nc{\eca}{\mathcal{E}} \nc{\wca}{\mathcal{W}} \nc{\ica}{\mathcal{I}}

\nc{\vp}{\varphi} \nc{\ddt}{\tfrac{d}{dt}} \nc{\dsdt}{\tfrac{d^2}{dt^2}} \nc{\dds}{\frac{d}{ds}}
\nc{\dpar}{\frac{\partial}{\partial t}} \nc{\im}{\mathrm{i}}

\nc{\SO}{\mathrm{SO}} \nc{\Spe}{\mathrm{Sp}} \nc{\Sl}{\mathrm{SL}}
\nc{\SU}{\mathrm{SU}} \nc{\Or}{\mathrm{O}} \nc{\U}{\mathrm{U}} \nc{\Gl}{\mathrm{GL}}
\nc{\Se}{\mathrm{S}} \nc{\Cl}{\mathrm{Cl}} \nc{\Spin}{\mathrm{Spin}}
\nc{\Pin}{\mathrm{Pin}} \nc{\G}{\mathrm{GL}_n(\RR)} \nc{\g}{\mathfrak{gl}_n(\RR)}

\nc{\RR}{{\Bbb R}} \nc{\HH}{{\Bbb H}} \nc{\CC}{{\Bbb C}} \nc{\ZZ}{{\Bbb Z}}
\nc{\FF}{{\Bbb F}} \nc{\NN}{{\Bbb N}} \nc{\QQ}{{\Bbb Q}} \nc{\PP}{{\Bbb P}} \nc{\OO}{{\Bbb O}}

\nc{\vs}{\vspace{.2cm}} \nc{\vsp}{\vspace{1cm}} \nc{\ip}{\langle\cdot,\cdot\rangle}
\nc{\ipp}{(\cdot,\cdot)} \nc{\la}{\langle} \nc{\ra}{\rangle} \nc{\unm}{\tfrac{1}{2}}
\nc{\unc}{\tfrac{1}{4}} \nc{\und}{\frac{1}{16}} \nc{\no}{\vs\noindent}
\nc{\lam}{\Lambda^2(\RR^n)^*\otimes\RR^n} \nc{\tangz}{{\rm T}^{\rm Zar}}
\nc{\nor}{{\sf n}}  \nc{\mum}{/\!\!/} \nc{\kir}{/\!\!/\!\!/}
\nc{\Ri}{\tfrac{4\Ric_{\mu}}{||\mu||^2}} \nc{\ds}{\displaystyle}
\nc{\ben}{\begin{enumerate}} \nc{\een}{\end{enumerate}} \nc{\f}{\frac}
\nc{\lb}{[\cdot,\cdot]} \nc{\isn}{\tfrac{1}{||v||^2}}
\nc{\gkp}{(\ggo=\kg\oplus\pg,\ip)} \nc{\ukh}{(\ug=\kg\oplus\hg,\ip)}
\nc{\tgkp}{(\tilde{\ggo}=\kg\oplus\pg,\ip)}
\nc{\wt}{\widetilde}
\nc{\iop}{\mathtt{i}} \nc{\jop}{\mathtt{j}} 
\nc{\Hk}{H_{\kil}} \nc{\gk}{g_{\kil}}

\nc{\Hess}{\operatorname{Hess}} \nc{\ad}{\operatorname{ad}}
\nc{\Ad}{\operatorname{Ad}} \nc{\rank}{\operatorname{rk}}
\nc{\Irr}{\operatorname{Irr}} \nc{\End}{\operatorname{End}}
\nc{\Aut}{\operatorname{Aut}} \nc{\Inn}{\operatorname{Inn}}
\nc{\Der}{\operatorname{Der}} \nc{\Ker}{\operatorname{Ker}}
\nc{\Iso}{\operatorname{Iso}} \nc{\Diff}{\operatorname{Diff}}
\nc{\Lie}{\operatorname{L}} \nc{\tr}{\operatorname{tr}} \nc{\dif}{\operatorname{d}}
\nc{\sen}{\operatorname{sen}} \nc{\modu}{\operatorname{mod}}
\nc{\CRic}{\operatorname{PP}} \nc{\Cric}{\operatorname{P}} \nc{\Ricci}{\operatorname{Ric}}
\nc{\sym}{\operatorname{sym}} \nc{\herm}{\operatorname{herm}} \nc{\symac}{\operatorname{sym^{ac}}}
\nc{\symc}{\operatorname{sym^{c}}} \nc{\scalar}{\operatorname{Sc}}
\nc{\grad}{\operatorname{grad}} \nc{\ricci}{\operatorname{Rc}} \nc{\kil}{\operatorname{B}} \nc{\cas}{\operatorname{C}} \nc{\lic}{\operatorname{L}}
\nc{\Nor}{\operatorname{Norm}}  \nc{\ricc}{\operatorname{Rc^{c}}}
\nc{\Ricc}{\operatorname{Ric^{c}}} \nc{\ricac}{\operatorname{Rc^{ac}}}
\nc{\Ricac}{\operatorname{Ric^{ac}}} \nc{\Riem}{\operatorname{Rm}} \nc{\Sec}{\operatorname{Sec}}
\nc{\riccig}{\operatorname{ric^{\gamma}}} \nc{\mm}{\operatorname{m}} \nc{\Mm}{\operatorname{M}}
\nc{\Le}{\operatorname{L}} \nc{\tang}{\operatorname{T}}
\nc{\level}{\operatorname{level}} \nc{\rad}{\operatorname{r}}
\nc{\abel}{\operatorname{ab}} \nc{\CH}{\operatorname{CH}} \nc{\Cone}{{\mathcal C}} \nc{\CCone}{\operatorname{CC}} \nc{\CP}{{\mathcal P}}
\nc{\mcc}{\operatorname{mcc}} \nc{\Adj}{\operatorname{Adj}}
\nc{\Order}{\operatorname{O}}  \nc{\inj}{\operatorname{inj}} \nc{\proy}{\operatorname{pr}}
\nc{\vol}{\operatorname{vol}} \nc{\Diag}{\operatorname{Dg}} \nc{\Diagg}{\operatorname{Diag}}
\nc{\Spec}{\operatorname{Spec}} \nc{\Ima}{\operatorname{Im}} \nc{\Rea}{\operatorname{Re}}
\nc{\spann}{\operatorname{span}} \nc{\Aff}{\operatorname{Aff}} \nc{\E}{\operatorname{E}} \nc{\id}{\operatorname{id}} \nc{\dete}{\operatorname{det}} \nc{\Crit}{\operatorname{Crit}} \nc{\val}{\operatorname{val}}

\theoremstyle{plain}
\newtheorem{theorem}{Theorem}[section]
\newtheorem{proposition}[theorem]{Proposition}
\newtheorem{corollary}[theorem]{Corollary}
\newtheorem{lemma}[theorem]{Lemma}

\theoremstyle{definition}

\theoremstyle{remark}
\newtheorem{remark}[theorem]{Remark}

\title{Einstein metrics on homogeneous spaces $H\times H/\Delta K$}

\author{Jorge Lauret}  
\author{Cynthia Will}

\address{FaMAF, Universidad Nacional de C\'ordoba and CIEM, CONICET (Argentina)}
\email{jorgelauret@unc.edu.ar} 
\email{cynthia.will@unc.edu.ar}

\thanks{This research was partially supported by grants from Univ.\ Nac.\ de C\'ordoba, FONCYT and CONICET (Argentina).  We would also like to acknowledge support from the ICTP through the Associates Programme and from the Simons Foundation through grant number 284558FY19.}

\date{\today}

\begin{document}

\maketitle

\begin{abstract}
Given any compact homogeneous space $H/K$ with $H$ simple, we consider the new space $M=H\times H/\Delta K$, where $\Delta K$ denotes diagonal embedding, and study the existence, classification and stability of $H\times H$-invariant Einstein metrics on $M$, as a first step into the largely unexplored case of homogeneous spaces of compact non-simple Lie groups.  We find unstable Einstein metrics on $M$ for most spaces $H/K$ such that their standard metric is Einstein (e.g., isotropy irreducible) and the Killing form of $\kg$ is a multiple of the Killing form of $\hg$ (e.g., $K$ simple), a class which contains $17$ families and $50$ individual examples.  A complete classification is obtained in the case when $H/K$ is an irreducible symmetric space and $K$ is simple.  We also study the behavior of the scalar curvature function on the space of all normal metrics on $M=H\times H/\Delta K$ (none of which is Einstein), obtaining that the standard metric is a global minimum.           
\end{abstract}

\tableofcontents

\section{Introduction}\label{intro}

Most of the literature on the existence and classification of Einstein metrics on a compact homogeneous space $M=G/K$ is devoted to the case when the Lie group $G$ is simple.  This is understandable considering the complicated structure of the isotropy representation of $M=G/K$ when $G$ has more than one simple factor, giving rise to a large and tricky space $\mca^G$ of $G$-invariant metrics and bringing about painful computations for the Ricci curvature.   

A first novelty we find in the case when $G$ is not simple, say with $s$ simple factors, is the presence of the $s$-parametric subspace $\mca^{norm}\subset\mca^G$ of all normal metrics on $M=G/K$ (i.e., defined by bi-invariant inner products on $\ggo$).  The existence and uniqueness of normal Einstein metrics is a natural open problem.  It is also natural the search for distinguished normal metrics on a given homogeneous space $M=G/K$, as for example critical points of the scalar curvature function $\scalar:\mca^{norm}_1\rightarrow\RR$, where $\mca^{norm}_1$ is the space of all unit volume normal metrics.  The potential relevance of the standard or Killing metric on $M=G/K$ from this variational viewpoint is particularly interesting.    

In this paper, we study the Ricci and scalar curvature of homogeneous spaces of the following form, which may be considered as a first step into the largely unexplored case of $G$ non-simple.  Given any $n$-dimensional homogeneous space $H/K$, where $H$ is a compact simple Lie group and $K$ a proper closed subgroup of $H$ with $d:=\dim{K}>0$, we consider the new compact homogeneous space 
\begin{equation}\label{HHK-def}
M^{2n+d}=G/\Delta K,  
\qquad \mbox{where}\quad G:=H\times H \quad \mbox{and}\quad \Delta K:=\{(k,k)\in G:k\in K\}.
\end{equation}
Let $\hg$, $\kg$ and $\ggo$ denote the corresponding Lie algebras.  If $\hg=\kg\oplus\qg$ is the $\kil_{\hg}$-orthogonal reductive decomposition for $H/K$, where $\kil_{\hg}$ denotes the Killing form of $\hg$, then $\ggo=\hg\oplus\hg=\Delta\kg\oplus\pg$ is the $\kil_{\ggo}$-orthogonal reductive decomposition for $M=G/\Delta K$, where $\Delta\kg=\{(Z,Z):Z\in\kg\}$ and
$$
\pg=\pg_1\oplus\pg_2\oplus\pg_3, 
\quad
\pg_1=(\qg,0), \quad \pg_2=(0,\qg), \quad \pg_3=\{(Z,-Z):Z\in\kg\}.
$$
We study $G$-invariant metrics on $M=G/\Delta K=H\times H/\Delta K$ of the form  
$$
g|_{\pg_i\times\pg_i}=x_i(-\kil_\ggo)|_{\pg_i\times\pg_i}, \; i=1,2,3, \quad g|_{\pg_1\times\pg_2}=x_4(-\kil_\ggo)|_{\pg_1\times\pg_2}, \quad g|_{\pg_1\times\pg_3}=g|_{\pg_2\times\pg_3}=0,
$$
where $x_1,x_2,x_3>0$ and $x_1x_2>x_4^2$, which are denoted by 
$$
g=(x_1,x_2,x_3,x_4).  
$$ 
These metrics exhaust $\mca^G$ if and only if the isotropy representation of $H/K$ is irreducible and of real type and $K$ is either simple or one-dimensional.  We call $g$ {\it diagonal} when $x_4=0$ and denote by $\mca^{diag}$ the space of all diagonal metrics.

Our first result concerns normal metrics on $M=G/\Delta K$.  

\begin{theorem}\label{thm1}
Let $M=G/\Delta K$ be a homogeneous space as defined in \eqref{HHK-def}.  Then, 
\hspace{1cm}
\begin{enumerate}[{\rm (i)}]
\item $\mca^{norm}\subset\mca^{diag}$ (see Lemma \ref{gz1z2}).  

\item A normal metric on $M=G/\Delta K$ is never Einstein (see Proposition \ref{normal-E}).  

\item The standard metric (i.e., defined by $-\kil_\ggo$) is a global minimum and the only critical point of the function $\scalar:\mca^{norm}_1\rightarrow\RR$ (see at the end of \S\ref{normal-sec}).  In particular, $\mca^{norm}$ is not invariant under the Ricci flow (see Figure \ref{fig}).  
\end{enumerate}
\end{theorem}

It is worth noting that part (i) of the above theorem does not longer hold for $M=H\times\dots\times H/\Delta K$ with three or more copies of $H$ and part (iii) is false for $M=G_1\times G_2/K$ with $G_1\not\simeq G_2$ (see \cite{Rical}).  

Secondly, we find a very large class of homogeneous spaces of the form $M=G/\Delta K$ admitting diagonal Einstein metrics and prove that they are all unstable as critical points of the Hilbert action (see \S\ref{diag-sec} and \S\ref{stab-sec}).  

\begin{theorem}\label{thm2}
\hspace{1cm}
\begin{enumerate}[{\rm (i)}] 
\item
The homogeneous space $M=G/\Delta K$ admits a diagonal Einstein metric if and only if the following conditions hold (see Theorem \ref{HHK-E}): 
\begin{enumerate}[{\rm (a)}] 
\item $\cas_\chi=\kappa I_\qg$ for some $\kappa\in\RR$ ($0<\kappa\leq\unm$), where $\cas_\chi$ is the Casimir operator of the isotropy representation of the homogeneous space $H/K$ with respect to $-\kil_{\hg}|_\kg$ (equivalently, the standard metric on $H/K$ is Einstein; these spaces were classified by Wang and Ziller in \cite{WngZll2}).    

\item $\kil_\kg=a\kil_\hg|_\kg$ for some $a\in\RR$ ($0\leq a<1$; e.g., if $K$ is simple or abelian).     

\item $(2\kappa+1)^2 \geq 8a(1-a+\kappa)$ (one has that $\kappa=\frac{d(1-a)}{n}$).  
\end{enumerate} 
In that case, $x_1=x_2$ and there exist exactly two non-homothetic Einstein metrics $g_1$ and $g_2$, unless $K$ is abelian or $H/K$ is $\Spe(24)/\Spe(8)^3$ or $\Spe(12)/\Spe(3)^4$, in which case there is only one (see Figure \ref{fig}).  The space $\mca^{diag}$ is Ricci flow invariant if and only if conditions (a) and (b) hold.

\item All these diagonal Einstein metrics are canonical variations of the submersion $M=H\times H/\Delta K\rightarrow H/K\times H/K$ with fiber the symmetric space $K\times K/\Delta K$ (see \S\ref{canvar}). 

\item The Einstein metric $g_1$ is $G$-unstable with coindex $\geq 2$ and a local minimum of $\scalar:\mca_1^{diag}\rightarrow\RR$, while $g_2$ is also $G$-unstable but it is a saddle point of $\scalar:\mca_1^{diag}\rightarrow\RR$ (see Proposition \ref{Lpe} and Figure \ref{fig}).  
\end{enumerate}
\end{theorem}

The homogeneous spaces $H/K$ satisfying conditions (a) and (b) have been listed in Tables \ref{table1}-\ref{table2-ss} given in Appendix \ref{table-sec}, there are $17$ families and $50$ sporadic examples.  In the last column of the tables is indicated whether inequality (c) holds or not for each space.  It turns out that only a few of them do not satisfy (c), providing a great amount of homogeneous Einstein metrics.   The case of $M=G_1\times G_2/K$ with $G_1\not\simeq G_2$ is studied in \cite{Es2}. 

Surprisingly, most of the non-existence cases correspond to irreducible symmetric spaces $H/K$, e.g., 
$$
\begin{array}{ll}
M^{20}=\SU(4)\times\SU(4)/\Delta\Spe(2), & M^{27}=\SO(7)\times\SO(7)/\Delta\SO(6), \\ 
M^{14}=\Spe(2)\times\Spe(2)/\Delta(\Spe(1)\times\Spe(1)), & M^{68}=F_4\times F_4/\Delta\SO(9),
\end{array}
$$
which leaded us to explore for non-diagonal Einstein metrics.  

As well known, the non-diagonal case requires heavy computations for the Ricci curvature (cf.\ the proof of Proposition \ref{ri}).  Nevertheless, we obtain a complete classification of $G$-invariant Einstein metrics on $M=G/\Delta K$ in the case when $H/K$ is an irreducible symmetric space with $K$ simple (see \S\ref{HHKsym-sec}).  Moreover, we found a non-diagonal metric on $M=G/\Delta K$ which is Einstein for any irreducible symmetric space $H/K$.    

\begin{theorem}\label{thm3}
Let $H/K$ be an irreducible symmetric space such that $\kil_\kg=a\kil_\hg|_\kg$, $a>0$ (see Table \ref{table1}).  Then the following are $G$-invariant Einstein metrics on the homogeneous space $M=G/\Delta K$ (see Theorem \ref{End}): 
\begin{enumerate}[{\rm (i)}]
\item For $a<\unm$, the two diagonal metrics $g_1=(x_+,x_+,1,0)$ and $g_2=(x_-,x_-,1,0)$ provided by Theorem \ref{thm2}, where 
$
x_\pm=\frac{1\pm \sqrt{1 - a(3-2a)}}{2a}.  
$ 
\item For $a>\unm$, the metric $g_3=(1,1,1,y)$, where 
$
y=\frac{1}{2}\sqrt{\frac{2a-1}{2-a}}.  
$
\item For any $a$, the metric $g_5=\left(\unm,\frac{3}{2},1,\unm\right)$.  
\end{enumerate}
Moreover, this gives a complete classification up to isometry and scaling of $G$-invariant Einstein metrics on $M=G/\Delta K$ in the case when $K$ is simple.  The metric $g_5$ is also Einstein on $M=G/\Delta K$ for any irreducible symmetric space $H/K$ (this is most likely the Einstein metric predicted by the Graph Theorem in \cite{BhmWngZll}; see Remark \ref{E2z-rem8}).   
\end{theorem}

The symmetric space $\SU(3)/\U(2)$ produces the example $M^{12}=\SU(3)\times\SU(3)/\Delta\U(2)$ given in \cite[Table 2]{BhmKrr} and more examples of this kind are given in \cite[Examples 5.13,5.14]{Bhm}.  We refer to \cite{Gtr, BRF, PdsRff2} for the study of generalized Einstein metrics on the spaces of the form $M=H\times H/\Delta K$. 

The paper is organized as follows.  After some preliminaries in \S\ref{preli} on the formulas for the Ricci curvature of homogeneous metrics that will be used, we compute in \S\ref{ijk-sec} the structural constants of $G/\Delta K$, providing an alternative proof for the formula of the Ricci eigenvalues.  Section \S\ref{normal-sec} is devoted to normal metrics and in \S\ref{diag-sec} we study the existence of diagonal Einstein metrics, whose stability type is obtained in \S\ref{stab-sec}.  Finally, in \S\ref{HHKsym-sec}, the non-diagonal case is considered.

\vs \noindent {\it Acknowledgements.}  We are very grateful with Christoph B\"ohm and the anonymous referee for many helpful comments and suggestions.

\section{Preliminaries}\label{preli}

\subsection{Ricci curvature of homogeneous spaces}\label{ric-sec}
Given a compact and connected differentiable manifold $M^n$ which is homogeneous, we fix an almost-effective transitive action of a compact connected Lie group $G$ on $M$.  The $G$-action determines a presentation $M=G/K$ of $M$ as a homogeneous space, where $K\subset G$ is the isotropy subgroup at some point $o\in M$.  Let $\mca^G$ denote the finite-dimensional manifold of all $G$-invariant Riemannian metrics on $M$.  For any reductive decomposition $\ggo=\kg\oplus\pg$ (i.e., $\Ad(K)\pg\subset\pg$), giving rise to the usual identification $T_oM\equiv\pg$, where $\ggo$ and $\kg$ are the Lie algebras of $G$ and $K$, respectively, we identify any $g\in\mca^G$ with the corresponding $\Ad(K)$-invariant inner product on $\pg$, also denoted by $g$.  

The Ricci tensor $\ricci(g):\pg\times\pg\rightarrow\RR$ of a metric $g\in\mca^G$ is given by
\begin{align}
\ricci(g)(X,Y) =& -\unm\sum_{i,j} g([X,X_i]_\pg,X_j)g([Y,X_i]_\pg,X_j) \label{Rc}\\ 
&+ \unc\sum_{i,j} g([X_i,X_j]_\pg,X)g([X_i,X_j]_\pg,Y) -\unm\kil_\ggo(X,Y), \qquad\forall X,Y\in\pg, \notag
\end{align} 
where $\{ X_i\}$ is any $g$-orthonormal basis of $\pg$, $\lb_\pg$ denotes the projection of the Lie bracket of $\ggo$ on $\pg$ relative to $\ggo=\kg\oplus\pg$ and $\kil_{\ggo}$ is the Killing form of the Lie algebra $\ggo$ (see \cite[7.38]{Bss}).  The Ricci operator $\Ricci(g):\pg\rightarrow\pg$ is therefore given by
\begin{equation}\label{Ric2}
\Ricci(g) = \Mm(g) - \unm\kil(g),
\end{equation}
where $\kil(g)$ is defined by $g(\kil(g)\cdot,\cdot):=\kil_{\ggo}|_{\pg\times\pg}$ and $\Mm(g)$ by
\begin{equation}\label{mm4}
g(\Mm(g) X,X) = -\unm\sum_{i,j} g([X,X_i]_\pg,X_j)^2+ \unc\sum_{i,j} g([X_i,X_j]_\pg,X)^2, \qquad\forall X\in\pg.
\end{equation} 

For each bi-invariant metric $g_b$ on $\ggo$, we consider the $g_b$-orthogonal reductive decomposition $\ggo=\kg\oplus\pg$ and the {\it normal} metric $g_b\in\mca^G$ determined by $g_b|_{\pg\times\pg}$.  It follows from \cite[Proposition (1.9)]{WngZll2} that its Ricci operator is given by 
\begin{equation}\label{ricgb}
\Ricci(g_b) = \unm\cas_{\pg,g_b|_\kg} + \unc\cas_{\ggo,g_b}|_\pg, 
\end{equation}
where 
$$
\cas_{\pg,g_b|_\kg}:\pg\longrightarrow\pg, \qquad \cas_{\pg,g_b|_\kg}:=-\sum_i\left(\ad{Z_i}|_\pg\right)^2,
$$ 
is the Casimir operator of the isotropy representation $\chi:K\rightarrow\End(\pg)$ of $M=G/K$ with respect to $g_b|_{\kg}$, $\{ Z_i\}$ is any $g_b$-orthonormal basis of $\kg$, 
$$
\cas_{\ggo,g_b}:\ggo\longrightarrow\ggo, \qquad \cas_{\ggo,g_b}:=-\sum_i\left(\ad{Z_i}\right)^2 -\sum_i\left(\ad{X_i}\right)^2,
$$ 
is the Casimir operator of the adjoint representation of $G$ relative to $g_b$ and $\cas_{\ggo,g_b}|_\pg$ is the map obtained by restricting and projecting $\cas_{\ggo,g_b}$ on $\pg$.  If $G$ is semisimple, $\ggo=\ggo_1\oplus\dots\oplus\ggo_s$ is the decomposition in simple ideals of $\ggo$ and $g_b=z_1(-\kil_{\ggo_1})+\dots+z_s(-\kil_{\ggo_s})$, $z_1,\dots,z_s>0$, then $\cas_{\ggo,g_b} = \tfrac{1}{z_1}I_{\ggo_1}+\dots+\tfrac{1}{z_s}I_{\ggo_s}$, where $I_{\ggo_i}$ denotes the identity map on $\ggo_i$.   

In particular, for the {\it standard} metric $\gk\in\mca^G$ determined by $-\kil_\ggo|_{\pg}$, 
\begin{equation}\label{ricgB}
\Ricci(\gk) = \unm\cas_{\chi} + \unc I_\pg, \qquad \Mm(\gk)= \unm\cas_{\chi} - \unc I_\pg, \qquad \kil(\gk)=-I_\pg,
\end{equation}
where 
$$
\cas_\chi:=\cas_{\pg,-\kil_\ggo|_\kg}:\pg\longrightarrow\pg
$$ 
is the Casimir operator of the isotropy representation with respect to $-\kil_\ggo|_{\kg}$.

\subsection{Homogeneous spaces $H\times H/\Delta K$}\label{HHK-sec}
We study in this paper $G$-invariant metrics on compact homogeneous spaces of the form 
$$
M^{2n+d}=G/\Delta K = H\times H/\Delta K, \qquad G:=H\times H,
$$
where $H$ is a compact connected simple Lie group and $K\subset H$ a proper closed subgroup diagonally embedded in $G$, that is, $\Delta K:=\{(k,k):k\in K\}\subset G$.  Here $\dim{H}=n+d$ and $\dim{K}=d>0$.  Note that $M$ is determined by the $n$-dimensional homogeneous space $H/K$, which is assumed to be almost-effective.  

If $\hg=\kg\oplus\qg$ is the $\kil_{\hg}$-orthogonal reductive decomposition for $H/K$, then 
$$
\ggo=\hg\oplus\hg=\Delta\kg\oplus\pg, \qquad \Delta\kg=\{(Z,Z):Z\in\kg\},
$$
is the $\kil_{\ggo}$-orthogonal reductive decomposition for $M=G/\Delta K$ (note that $\kil_{\ggo}=\kil_{\hg}+\kil_{\hg}$), where  
\begin{equation}\label{gBdec}
\pg=\pg_1\oplus\pg_2\oplus\pg_3, 
\end{equation}
and
$$
\pg_1=(\qg,0), \qquad \pg_2=(0,\qg), \qquad \pg_3=\{(Z,-Z):Z\in\kg\}.
$$
We obtain that $\pg=\pg_1\oplus\pg_2\oplus\pg_3$ is a $\gk$-orthogonal $\Ad(K)$-invariant decomposition, where $\gk$ is the standard metric on $G/\Delta K$, and as $\Ad(K)$-representations, $\pg_1\simeq\pg_2\simeq\qg$, the isotropy representation of $H/K$ (so $G/\Delta K$ is never multiplicity-free), and $\pg_3\simeq\kg$, the adjoint representation of $K$.  In particular, $\pg_1$ and $\pg_2$ are $\Ad(K)$-irreducible if and only if $H/K$ is an isotropy irreducible homogeneous space and $\pg_3$ is $\Ad(K)$-irreducible if and only if $K$ is either simple or one-dimensional.  In that case, $\dim{\mca^G}=4,5,7$, depending on the type of the $K$-representation $\qg$ (i.e., $\End_K(\qg)=\RR,\CC,\HH$, see \cite[Chapter II]{Wlf}).  

The Lie brackets between the $\pg_i$'s satisfy: 
\begin{align}
[\pg_1,\pg_1]\subset\pg_1+\pg_3+\Delta\kg, \quad [\pg_2,\pg_2]\subset\pg_2+\pg_3+\Delta\kg, 
\quad [\pg_1,\pg_2]=0, \label{pipj}\\ 
[\pg_3,\pg_3]\subset\Delta\kg, \quad [\pg_3,\pg_1]\subset\pg_1, \quad [\pg_3,\pg_2]\subset\pg_2. \notag
\end{align}
We also consider the $\gk$-orthogonal $\Ad(K)$-invariant decomposition
$$
\pg_3=\pg_3^0\oplus\dots\oplus\pg_3^t, \qquad \pg_3^l=\{(Z,-Z):Z\in\kg_l\}, 
$$
where 
\begin{equation}\label{decs} 
\kg=\kg_0\oplus\kg_1\oplus\dots\oplus\kg_t, 
\end{equation} 
$\kg_0$ is the center of $\kg$ and $\kg_1,\dots\kg_t$ are the simple ideals of $\kg$.  

We denote by 
\begin{equation}\label{cas}
\cas_{\chi}:\qg\rightarrow\qg
\end{equation} 
the Casimir operator of the isotropy representation of the homogeneous space $H/K$ with respect to $-\kil_{\hg}|_\kg$, and by 
$$
g=(x_1,x_2,x_3)_{\gk}, \qquad x_1,x_2,x_3>0,
$$ 
the $G$-invariant metric on $M$ defined by the $\Ad(\Delta K)$-invariant inner product 
$$
x_1(-\kil_\ggo)|_{\pg_1}+x_2(-\kil_\ggo)|_{\pg_2}+x_3(-\kil_\ggo)|_{\pg_3}.
$$
A metric of this form will be called {\it diagonal}, in the sense that it is diagonal with respect to the decomposition \eqref{gBdec}.   

\begin{proposition}\label{ricHHK}\cite[Proposition 3.2]{BRF}
The Ricci operator of the diagonal metric 
$
g=(x_1,x_2,x_3)_{\gk}
$ 
on the homogeneous space $M=H\times H/\Delta K$ is given by, 
\begin{enumerate}[{\rm (i)}]
\item $\Ricci(g)|_{\pg_1} = 
\frac{1}{2x_1}\left(1 - \frac{x_3}{2x_1}\right) \cas_{\chi}
+ \frac{1}{4x_1}I_{\pg_1}$.  
\item[ ]
\item $\Ricci(g)|_{\pg_2} = 
\tfrac{1}{2x_2}\left(1 - \tfrac{x_3}{2x_2}\right) \cas_{\chi}
+ \tfrac{1}{4x_2}I_{\pg_2}$.  
\item[ ]
\item The decomposition $\pg=\pg_1\oplus\pg_2\oplus\pg_3^0\oplus\dots\oplus\pg_3^t$ is $\ricci(g)$-orthogonal.
\item[ ]
\item $\Ricci(g)|_{\pg_3^l}= r_{3,l}I_{\pg_3^l}$, where for each $l=0,1,\dots,t$,
\begin{align*}
r_{3,l} := a_l\left(\frac{1}{2x_3}-\frac{x_3}{8}\left(\frac{1}{x_1^2} +\frac{1}{x_2^2}\right)  \right)  
+\frac{x_3}{8}\left(\frac{1}{x_1^2} + \frac{1}{x_2^2}\right), \qquad \kil_{\kg_l}=a_l\kil_\hg|_{\kg_l}.
\end{align*}  
\end{enumerate}
\end{proposition}

\begin{remark}\label{RF}
The space $\mca^{diag}$ is therefore Ricci flow invariant if and only if $\cas_\chi=\kappa I_\qg$ for some $\kappa\in\RR$ and $\kil_\kg=a\kil_\hg|_\kg$ for some $a\in\RR$ (i.e., either $a=0$ and $K$ is abelian or $K$ is semisimple and $a_1=\dots=a_t=a$).   
\end{remark}

\begin{remark}\label{gB-E}
It follows that the standard metric $\gk$ (i.e., $x_1=x_2=x_3=1$) is Einstein if and only if $\kg_0=0$ (i.e., $K$ semisimple), $\kil_\kg=a\kil_\hg|_\kg$ and $\cas_\chi=a I_\qg$.  In Tables \ref{table1}-\ref{table2-ss} (see Appendix \ref{table-sec}) one can easily check that this is impossible (cf.\ \S\ref{Est-sec} below), so $\gk$ is never Einstein on $M=H\times H/\Delta K$.  
\end{remark}

\begin{proof}[Proof of Proposition \ref{ricHHK}]
We use the notation and the formula for the Ricci curvature of aligned homogeneous spaces given in \cite[Section 3]{BRF}.  It follows from \cite[Definition 2.1]{BRF} that $M=H\times H/\Delta K$ is aligned with $c_1=c_2=2$ and $2\lambda_l=a_l$ for $l=1,\dots,t$.  Since we are considering $g_b=\gk$, i.e., $z_1=z_2=1$, we have that $A_3=-1$, $B_3=1$ and $\cas_{\chi_1}=\cas_{\chi_2}=\cas_\chi$.  The proposition is therefore a direct application of the formulas given in  \cite[Proposition 3.2]{BRF}.   
\end{proof}

We note that in spite of $\pg_1\simeq\pg_2$ as $\Ad(K)$-representations, $\ricci(g)(\pg_1,\pg_2)=0$ for any diagonal metric $g$.

\subsection{Standard Einstein metrics}\label{Est-sec} 
It follows from Proposition \ref{ricHHK}, (i) and (ii) that a necessary condition for the existence of a diagonal Einstein metric on $M=H\times H/\Delta K$ is that $\cas_\chi=\kappa I_\qg$ for some $\kappa\in\RR$, which is equivalent by \eqref{ricgB} to have that the standard metric on $H/K$ is Einstein.   

Compact homogeneous spaces $H/K$ with $H$ simple such that the standard metric $\gk$ is Einstein were classified by Wang and Ziller in \cite{WngZll2}: beyond irreducible symmetric spaces ($7$ infinite families and $12$ individual examples, see \cite[7.102]{Bss}) and non-symmetric isotropy irreducible spaces ($10$ infinite families and $33$ individual examples, see \cite[7.106, 7.107]{Bss} or \cite[Tables 1,3,4]{Sch}), there are, with $H$ classical, $10$ infinite families parametrized by the natural numbers, $2$ conceptual constructions and $2$ isolated examples, and with $H$ exceptional, $20$ isolated examples.  The isotropy $K$ is either abelian or semisimple for most of them, we refer to \cite[Tables 1-4]{stab} (or  \cite[7.108, 7.109]{Bss}) for further information on the non-isotropy irreducible case.  Summarizing, the class of all homogeneous spaces $H/K$ such that $H$ is compact simple and $\gk$ is Einstein is quite large, it consists of a total of 29 infinite families and 67 individual spaces.  

If the standard metric $\gk$ on $H/K$ is Einstein with $\Ricci(\gk)=\rho I_{\pg}$ and so $\cas_{\chi}=\kappa I_{\qg}$ for some $\kappa\in\RR$ (see \eqref{cas}), where $\hg=\kg\oplus\qg$ is the $\kil_\hg$-orthogonal reductive decomposition, then the following holds:
\begin{enumerate}[{\small $\bullet$}]
\item $\kappa=2\rho-\unm$ and $0<\kappa\leq \unm$ (i.e., $\unc<\rho\leq\unm$), where equality $\kappa=\unm$ (i.e., $\rho=\unm$) holds if and only if $H/K$ is an irreducible symmetric space (i.e., $[\pg,\pg]\subset\kg$).  

\item If $\kg=\kg_0\oplus\kg_1\oplus\dots\oplus\kg_t$ is as in \eqref{decs}, then according to \cite[(10)]{stab},
\begin{equation*}
\kappa=\tfrac{1}{n}\sum_{l=0}^t (1-a_l)\dim{\kg_l}, \qquad \kil_{\kg_l}=a_l\kil_\hg|_{\kg_l}.  
\end{equation*}
Note that $a_0=0$.  

\item In particular, if $\kg$ is abelian then $\kappa=\tfrac{d}{n}$ and if $\kg$ is semisimple and $a_1=\dots=a_t=:a$, i.e., $\kil_\kg=a\kil_\hg|_\kg$ (e.g., if $\kg$ is simple), then 
\begin{equation}\label{kap}
\kappa=\tfrac{(1-a)d}{n}. 
\end{equation} 
\end{enumerate}

In Tables \ref{table1}-\ref{table2-ss} (see Appendix \ref{table-sec}), the homogeneous spaces $M^n=H/K$ with $H$ simple such that the standard metric $\gk$ is Einstein and $\kil_\kg=a\kil_\hg|_\kg$ for some $a\in\RR$ have been listed, including the values of $a$, $\kappa$ and $\rho$ for each space.  This class consists of $17$ families and $50$ isolated spaces.  

\begin{remark}
The classification of standard Einstein metrics is still open for $H$ non-simple (see \cite[Section 4.14]{NknRdnSlv}). 
\end{remark}

\section{Structural constants}\label{ijk-sec}

We provide in this section an alternative proof for the Ricci curvature formulas of diagonal metrics on the homogeneous spaces $M=H\times H/\Delta K$ given in Proposition \ref{ricHHK}.  

Given any homogeneous space $G/K$ and the $Q$-orthogonal reductive decomposition $\ggo=\kg\oplus\pg$ with respect to a bi-invariant inner product $Q$ on $\ggo$, the so called {\it structural constants} of a $Q$-orthogonal decomposition
$
\pg=\pg_1\oplus\dots\oplus\pg_r
$ 
in $\Ad(K)$-invariant  subspaces (not necessarily $\Ad(K)$-irreducible) are defined by
\begin{equation}\label{ijk-gen}
[ijk]:=\sum_{\alpha,\beta,\gamma} Q([e_{\alpha}^i,e_{\beta}^j], e_{\gamma}^k)^2,
\end{equation}
where $\{ e_\alpha^i\}$, $\{ e_{\beta}^j\}$ and $\{ e_{\gamma}^k\}$ are $Q$-orthonormal basis of $\pg_i$, $\pg_j$ and $\pg_k$, respectively.  Note that $[ijk]$ is invariant under any permutation of $ijk$ and it vanishes if and only if $[\pg_i,\pg_j]\perp\pg_k$.  It is easy to see using \eqref{mm4} that if we assume that $-\kil_\ggo|_{\pg_k\times\pg_k}=b_kQ$, $b_k\in\RR$ and a $G$-invariant metric on $G/K$ of the form 
\begin{equation}\label{metric}
g=x_1Q|_{\pg_1}+\dots+x_rQ|_{\pg_r}, \qquad x_i>0,
\end{equation}
satisfies that $\Ricci(g)|_{\pg_k}=\rho_kI_{\pg_k}$ for $k=1,\dots,r$, then 
\begin{align}
\rho_k=& \frac{b_k}{2x_k} -\frac{1}{4n_k}\sum_{i,j} [ijk]\left(\frac{x_i}{x_jx_k} +\frac{x_j}{x_ix_k} - \frac{x_k}{x_ix_j}\right), \qquad n_k:=\dim{\pg_k}.\label{rick}  
\end{align}

We resume the study of $M=G/\Delta K=H\times H/\Delta K$ and assume here that 
\begin{equation}\label{assum}
\cas_\chi=\kappa I_\pg \quad \mbox{and}\quad \kil_\kg=a\kil_\hg|_\kg \qquad \mbox{for some}\quad\kappa,a\in\RR,
\end{equation}
which are necessary conditions for the existence of a diagonal Einstein metric on $M$.    

\begin{lemma}\label{ijk}
The nonzero structural constants of the $\gk$-orthogonal reductive complement $\pg=\pg_1\oplus\pg_2\oplus\pg_3$ are given by, 
$$
[111]=[222]=(1-2\kappa)n 
=n-2(1-a)d, \qquad
[113]= [223]=\frac{\kappa n}{2} =  \frac{(1-a)d}{2}. 
$$
\end{lemma}

\begin{proof}
Consider the $\gk$-orthonormal bases $\{ e^1_\alpha=(X_\alpha,0)\}$, $\{ e^2_\alpha=(0,X_\alpha)\}$ and $\{ e^3_\alpha=(Z_\alpha,-Z_\alpha)\}$ of $\pg_1$, $\pg_2$ and $\pg_3$, respectively, where $\{ X_\alpha\}$ is a $(-\kil_\hg)$-orthonormal basis of $\qg$ and $\{ Z_\alpha\}$ is a basis of $\kg$ such that $-\kil_\hg(Z_\alpha,Z_\beta)=\delta_{\alpha\beta}\unm$.   

According to \eqref{mm4} and \eqref{ricgB}, for $i=1,2$, 
$$
[iii] = \sum_{\alpha,\beta,\gamma} \gk([e^i_\alpha,e^i_\beta],e^i_\gamma)^2 
= -4\tr{\Mm(\overline{g}_{\kil})} = -2\tr{\cas_{\chi}}+\tr{I_{\pg_i}} = (1-2\kappa)n,
$$
where $\overline{g}_{\kil}$ is the standard metric on $H/K$, and on the other hand,  
$$
[ii3] = \sum_{\alpha,\beta,\gamma} \gk([e^3_\alpha,e^i_\beta],e^i_\gamma)^2 
= \sum_{\alpha,\beta,\gamma} \gk([Z_\alpha,e^i_\beta],e^i_\gamma)^2 
= \sum_{\alpha} -\tr{(\ad{Z_\alpha})^2} \\
=  \tfrac{\tr{\cas_{\chi}}}{2} = \tfrac{\kappa n}{2},
$$
concluding the proof.  
\end{proof} 

\begin{corollary}\label{rics2str}
The Ricci curvature of the metric $g=(x_1,x_2,x_3)_{\gk}$ satisfies that 
$\Ricci(g)|_{\pg_i} = r_iI_{\pg_i}$, where 
$$
r_1
= \frac{1+2\kappa}{4x_1} - \frac{\kappa x_3}{4x_1^2}, \qquad 
r_2
= \frac{1+2\kappa}{4x_2} - \frac{\kappa x_3}{4x_2^2},  \qquad
r_3 = \frac{a}{2x_3} + \frac{(1-a)x_3}{8x_1^2} + \frac{(1-a)x_3}{8x_2^2}.
$$
\end{corollary}

\begin{remark}
It is easy to check that these formulas coincide with those in Proposition \ref{ricHHK} under the assumption \eqref{assum} made in this section.  
\end{remark}

\begin{proof}
We use the well-known formula for the Ricci eigenvalues in terms of structural constants (see e.g.\ \cite[(18)]{stab-dos}) to obtain that 
\begin{align*}
r_1 =& \tfrac{1}{2x_1} - \tfrac{1}{4n}[111]\tfrac{1}{x_1} - \tfrac{1}{2n}[131]\tfrac{x_3}{x_1^2} 
= \left(\tfrac{1}{2} - \tfrac{1-2\kappa}{4}\right)\tfrac{1}{x_1} - \tfrac{\kappa}{4}\tfrac{x_3}{x_1^2} 
= \tfrac{1+2\kappa}{4}\tfrac{1}{x_1} - \tfrac{\kappa}{4}\tfrac{x_3}{x_1^2}, \\ 
r_2 =& \tfrac{1}{2x_2} - \tfrac{1}{4n}[222]\tfrac{1}{x_2} - \tfrac{1}{2n}[232]\tfrac{x_3}{x_2^2} 
= \left(\tfrac{1}{2} - \tfrac{1-2\kappa}{4}\right)\tfrac{1}{x_2} - \tfrac{\kappa}{4}\tfrac{x_3}{x_2^2} 
= \tfrac{1+2\kappa}{4}\tfrac{1}{x_2} - \tfrac{\kappa}{4}\tfrac{x_3}{x_2^2}, \\ 
r_3 =& \tfrac{1}{2x_3} - \tfrac{1}{4d}[113]\left(\tfrac{2}{x_3}-\tfrac{x_3}{x_1^2}\right) 
- \tfrac{1}{4d}[223]\left(\tfrac{2}{x_3}-\tfrac{x_3}{x_2^2}\right)- \tfrac{1}{4d}[333]\tfrac{1}{x_3} \\ 
=& \left(\tfrac{1}{2} - \tfrac{1}{2d}[113] - \tfrac{1}{2d}[223] - \tfrac{1}{4d}[333]\right)\tfrac{1}{x_3} 
+\tfrac{1}{4d}[113]\tfrac{x_3}{x_1^2} + \tfrac{1}{4d}[223]\tfrac{x_3}{x_2^2}\\ 
=& \left(\tfrac{1}{2} - \tfrac{1-a}{4} - \tfrac{1-a}{4}\right)\tfrac{1}{x_3} 
+ \tfrac{1-a}{8}\tfrac{x_3}{x_1^2} + \tfrac{1-a}{8}\tfrac{x_3}{x_2^2}
= \tfrac{4-2(1-a)-2(1-a)}{8} \tfrac{1}{x_3} 
+ \tfrac{1-a}{8}\tfrac{x_3}{x_1^2} + \tfrac{1-a}{8}\tfrac{x_3}{x_2^2}, 
\end{align*}
concluding the proof.  
\end{proof}

\section{Normal metrics}\label{normal-sec}

As far as we know, the only known examples of normal Einstein metrics which are not standard are given on homogeneous spaces of the form $H\times K/\Delta K$, where $K\subset H$ are compact simple Lie groups, $g_b=-\kil_\hg+\tfrac{2a+1}{a}(-\kil_\kg)$ and $\kil_\kg=a\kil_\hg|_\kg$ (see \cite[pp.628]{WngZll2} or \cite[Section 5]{stab-killing}).  This normal metric is actually isometric to the Killing metric on the Lie group $H$.  

We study in this section normal metrics on the homogeneous space $M=H\times H/\Delta K$ given in \S\ref{HHK-sec}.  For each bi-invariant metric $g_b$ on $\ggo=\hg\oplus\hg$, say
\begin{equation*}
g_b=z_1(-\kil_{(\hg,0)})+z_2(-\kil_{(0,\hg)}), \qquad z_1,z_2>0,   
\end{equation*}
we consider the $g_b$-orthogonal reductive decomposition $\ggo=\Delta\kg\oplus\pg(g_b)$ and the $g_b$-orthogonal $\Ad(K)$-invariant decomposition 
\begin{equation*}
\pg(g_b)=\pg_1\oplus\pg_2\oplus\pg_3(g_b), \qquad \pg_3(g_b) := \left\{ \left(Z,-\tfrac{z_1}{z_2}Z\right):Z\in\kg\right\}.  
\end{equation*}
Every normal metric $g_b$ determines a $3$-parameter family of $G$-invariant metrics on $M$ defined by the $\Ad(\Delta K)$-invariant inner products on $\pg(g_b)$ given by
$$
x_1(-\kil_\ggo|_{\pg_1})+x_2(-\kil_\ggo|_{\pg_2})+x_3(-\kil_\ggo|_{\pg_3(g_b)}), \qquad x_1,x_2,x_3>0,
$$
which will be denoted by $\left(x_1,x_2,x_3\right)_{g_b}$.  We now show that all these metrics are also diagonal metrics with respect to the standard metric $\gk$ (i.e., $z_1=z_2=1$).  It is proved in \cite{Rical} that this is not longer true for homogeneous spaces $M=G/\Delta K$ with $G=H\times\dots\times H$, $s$-times, $s\geq 3$.  

\begin{lemma}\label{gz1z2}
For any $z_1,z_2,x_1,x_2,x_3>0$, 
$$
\left(x_1,x_2,x_3\right)_{g_b} = \left(z_1x_1,z_2x_2,\frac{2z_1z_2x_3}{z_1+z_2}\right)_{\gk}.
$$
\end{lemma}

\begin{proof} 
Let $\ggo=\Delta\kg\oplus\pg(g_b)$ and $\ggo=\Delta\kg\oplus\pg$ be the orthogonal reductive decompositions with respect to $g_b$ and $\gk$, respectively, and let $\hat{g}$ denote the $\Ad(K)$-invariant inner product on $\pg$ identified with the metric $g:=\left(x_1,x_2,x_3\right)_{g_b}\in\mca^G$.  If $T:\ggo\rightarrow\pg(g_b)$ is the $g_b$-orthogonal projection, then 
\begin{equation}\label{gz1}
\hat{g}(X,Y) = g(TX,TY), \qquad\forall X,Y\in\pg.  
\end{equation}
Using that $g_b|_{\Delta\kg}=\tfrac{z_1+z_2}{2}(-\kil_\ggo)|_{\Delta\kg}$ and the orthogonal decompositions 
$$
\pg=\pg_1\oplus\pg_2\oplus\pg_3, \qquad  \pg(g_b)=\pg_1\oplus\pg_2\oplus\pg_3(g_b),
$$
where $\pg_3:=\{(Z,-Z):Z\in\kg\}$, we obtain that $T|_{\pg_1\oplus\pg_2}=I$ and 
$$
T(Z,-Z)=\tfrac{2z_2}{z_1+z_2}(Z,-\tfrac{z_1}{z_2}Z), \qquad\forall Z\in\kg,
$$
since 
$$
(Z,-Z) = \tfrac{z_1-z_2}{z_1+z_2}(Z,Z) + 
\tfrac{2z_2}{z_1+z_2}(Z,-\tfrac{z_1}{z_2}Z).  
$$
It now follows from \eqref{gz1} that $\pg_1$, $\pg_2$ and $\pg_3$ are $\hat{g}$-orthogonal, that is, $\hat{g}=\left(y_1,y_2,y_3\right)_{\gk}$, 
\begin{align*}
y_1\gk(X,X)=\hat{g}(X,X) =& x_1g_b(X,X) = x_1z_1\gk(X,X), \qquad\forall X\in\pg_1,\\ 
y_2\gk(X,X)=\hat{g}(X,X) =& x_2g_b(X,X) = x_2z_2\gk(X,X), \qquad\forall X\in\pg_2, 
\end{align*}
and for any $X=(Z,-Z)\in\pg_3$, $Z\in\kg$,
\begin{align*}
\hat{g}(X,X) =& g(TX,TX) = \tfrac{4z_2^2}{(z_1+z_2)^2}x_3g_b\left((Z,-\tfrac{z_1}{z_2}Z),(Z,-\tfrac{z_1}{z_2}Z)\right) \\ 
=&\tfrac{4z_2^2}{(z_1+z_2)^2}x_3\left(z_1+z_2\tfrac{z_1^2}{z_2^2}\right)\gk(Z,Z)
=\tfrac{4z_1z_2}{z_1+z_2}x_3\gk(Z,Z).
\end{align*}
On the other hand, 
$$
\hat{g}(X,X) = y_3\gk(X,X) = 2y_3\gk(Z,Z),
$$
which gives $y_3=\tfrac{2z_1z_2}{z_1+z_2}x_3$, concluding the proof.  
\end{proof}

\begin{proposition}\label{normal-E}
A normal metric on $M=H\times H/\Delta K$ is never Einstein.  
\end{proposition}

\begin{remark}
In \cite{Rical}, it is shown that normal metrics are never Einstein on the much larger class of all aligned homogeneous spaces, i.e., $M=G_1\times\dots\times G_s/K$ with $G_i$ simple such that $b_3(M)=s-1$.  
We do not know the answer to the following natural question: Are there non-standard normal Einstein metrics other than the ones described on $H\times K/\Delta K$ at the beginning of the section?
\end{remark}

\begin{proof}
According to Proposition \ref{ricHHK} and Lemma \ref{gz1z2}, if the normal metric 
\begin{equation}\label{gb}
g_b=\left(z_1,z_2,\tfrac{2z_1z_2}{z_1+z_2}\right)_{\gk}
\end{equation} 
is Einstein, then $a_1=\dots=a_t=:a$ for some $a\geq 0$, i.e., $\kil_\kg=a\kil_\hg|_\kg$ (since the factor multiplying $a_l$ in the formula for $r_{3,l}$ is nonzero), and $\cas_\chi=\kappa I_\qg$ for some $\kappa\in\RR$ (note that $\kg$ is abelian if and only if $a=0$).  Now $r_1=r_2$ implies that $z_1=z_2$ and from $r_1=r_3$ it follows that $a=\kappa$, which is impossible by an easy inspection of 
Tables \ref{table1}-\ref{table2-ss} (cf.\ Remark \ref{gB-E}).
\end{proof}

It follows from \eqref{gb} that the automorphism of $G$ interchanging the copies of $H$ determines an isometry between the two normal metrics $g_b(z_1,z_2)$ and $g_b(z_2,z_1)$.  

The non-existence of Einstein metrics in the $2$-dimensional submanifold $\mca^{norm}\subset\mca^G$ of all normal metrics on $M=H\times H/\Delta K$ motivates the following natural questions: how does the scalar curvature function $\scalar$ behaves on $\mca^{norm}$?; is the standard metric $\gk$ an special point?; is $\mca^{norm}$ invariant under the Ricci flow?

\begin{lemma}\label{scal-normal}
The scalar curvature of a normal metric $g_b=\left(z_1,z_2,\tfrac{2z_1z_2}{z_1+z_2}\right)_{\gk}$ on $M=H\times H/\Delta K$ is given by 
$$
\scalar(g_b)= \frac{(d+n)z_1^2+2(2d+n-S)z_1z_2+(d+n)z_2^2}{4z_1z_2(z_1+z_2)}, \qquad S:=\sum_{l=1}^t a_ld_l, 
\quad d_l:=\dim{\kg_l}, 
$$
and the normalized scalar curvature $\scalar_N(g_b):=\scalar(g_b)\left(\det_{\gk}{g_b}\right)^{\frac{1}{\dim{M}}}$ by 
$$
\scalar_N(g_b)=\frac{(d+n)z_1^2+2(2d+n-S)z_1z_2+(d+n)z_2^2}{4z_1z_2(z_1+z_2)} z_1^{\tfrac{n}{2n+d}} z_2^{\tfrac{n}{2n+d}} \left(\frac{2z_1z_2}{z_1+z_2}\right)^{\tfrac{d}{2n+d}}. 
$$
\end{lemma}

\begin{proof}
Using Proposition \ref{ricHHK} and the fact that 
\begin{equation}\label{trcas}
\tr{\cas_{\chi}}=\sum_{l=0}^t (1-a_l)d_l
\end{equation} 
(see \eqref{cas} and \cite[(9)]{stab}), we obtain 
\begin{align*}
\scalar(g_b)
=& \tfrac{1}{2z_1}\left(1 - \tfrac{z_2}{z_1+z_2}\right) \sum_{l=0}^t (1-a_l)d_l + \tfrac{n}{4z_1} 
+\tfrac{1}{2z_2}\left(1 - \tfrac{z_1}{z_1+z_2}\right) \sum_{l=0}^t (1-a_l)d_l + \tfrac{n}{4z_2}\\ 
&+\left(\tfrac{z_1+z_2}{4z_1z_2}-\tfrac{z_1z_2}{4(z_1+z_2)}\left(\tfrac{1}{z_1^2} +\tfrac{1}{z_2^2}\right)\right) \sum_{l=1}^t d_la_l  
+d\tfrac{z_1z_2}{4(z_1+z_2)}\left(\tfrac{1}{z_1^2} + \tfrac{1}{z_2^2}\right) \\
=& \tfrac{1}{2(z_1+z_2)} \sum_{l=0}^t (1-a_l)d_l + \tfrac{n}{4z_1} 
+\tfrac{1}{2(z_1+z_2)}\sum_{l=0}^t (1-a_l)d_l + \tfrac{n}{4z_2}\\
&+\left(\tfrac{z_1+z_2}{4z_1z_2}-\tfrac{z_1^2+z_2^2}{4z_1z_2(z_1+z_2)}  \right) \sum_{l=1}^t d_la_l  
+d\tfrac{z_1^2+z_2^2}{4z_1z_2(z_1+z_2)}  \\ 
=& \tfrac{1}{z_1+z_2} \sum_{l=0}^t (1-a_l)d_l + \tfrac{n}{4z_1} + \tfrac{n}{4z_2} 
+\tfrac{1}{2(z_1+z_2)} \sum_{l=1}^t d_la_l 
+d\tfrac{z_1^2+z_2^2}{4z_1z_2(z_1+z_2)}  \\
=& d\tfrac{1}{z_1+z_2} + \tfrac{n}{4z_1} + \tfrac{n}{4z_2} 
-\tfrac{1}{2(z_1+z_2)}S  
+d\tfrac{z_1^2+z_2^2}{4z_1z_2(z_1+z_2)}  
= \tfrac{(d+n)z_1^2+2(2d+n-S)z_1z_2+(d+n)z_2^2}{4z_1z_2(z_1+z_2)},
\end{align*}
concluding the proof.  
\end{proof}

If $z_1=z$ and $z_2=\tfrac{1}{z}$, then it is straightforward to see that the function 
$$
f(z):=\scalar_N(g_b)=\tfrac{(d+n)z^4+2(2d+n-S)z^2+(d+n)}{4z(z^2+1)} \left(\tfrac{2z}{z^2+1}\right)^{\tfrac{d}{2n+d}} 
$$
has a unique critical point at $z=1$, i.e., at $\gk$, with second derivative 
$$
f''(1)=-\tfrac{(d+n)(d-2n-S)}{4n+2d}. 
$$
Since $d-S=\tr{\cas_\chi}<\unm n$ by \eqref{trcas} and \eqref{ricgB}, respectively, we have that $f''(1)>0$.  It is also easy to see that $f(z)$ converges to $\infty$ as $z\to 0,\infty$.  Thus the standard metric $\gk$ is a global minimum of 
$$
\scalar:\mca^{norm}_1\rightarrow\RR, 
$$ 
where $\mca^{norm}_1$ is the space of all unit volume normal metrics on $M=H\times H/\Delta K$.  Since $\gk$ is not Einstein, this implies that $\mca^{norm}$ is not invariant under the Ricci flow (see Figure \ref{fig}).

\begin{figure}\label{sec}
\includegraphics[width=49mm]{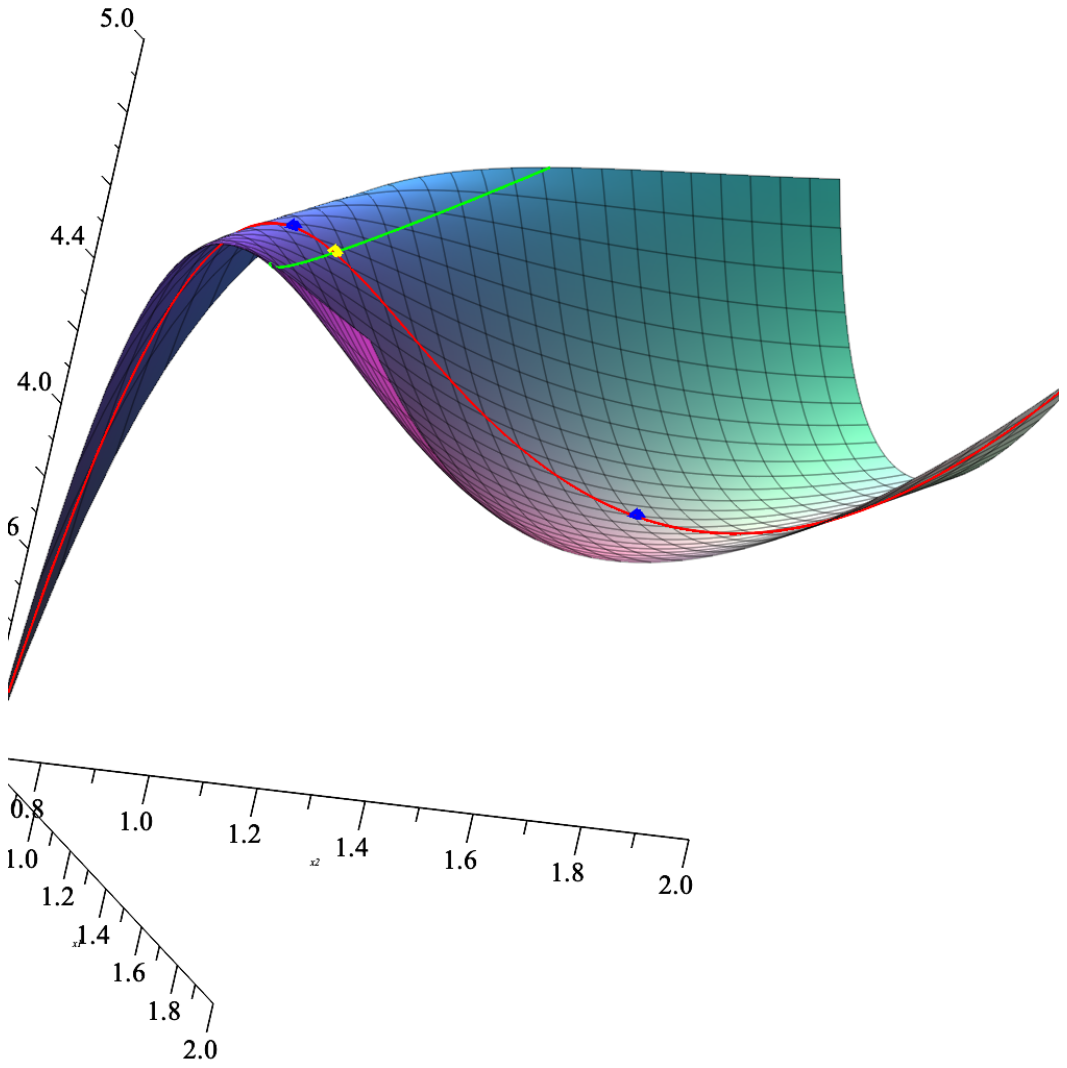}
\includegraphics[width=49mm]{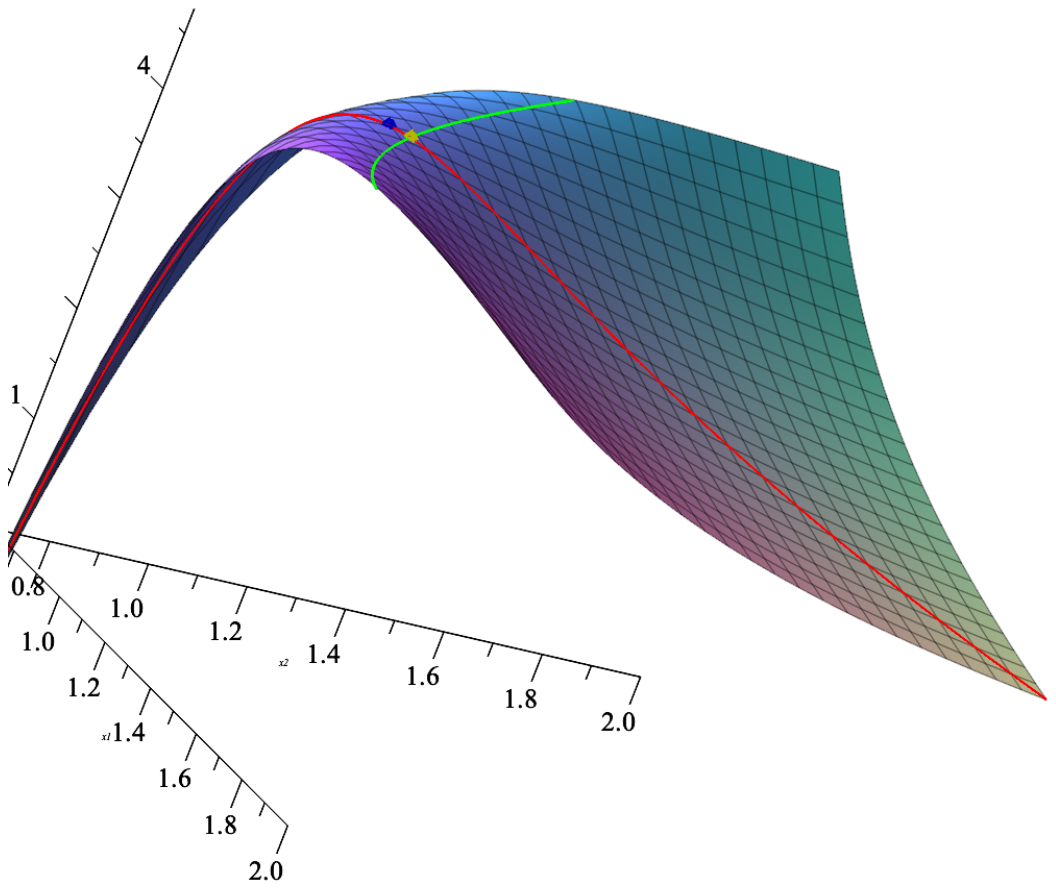}
\includegraphics[width=49mm]{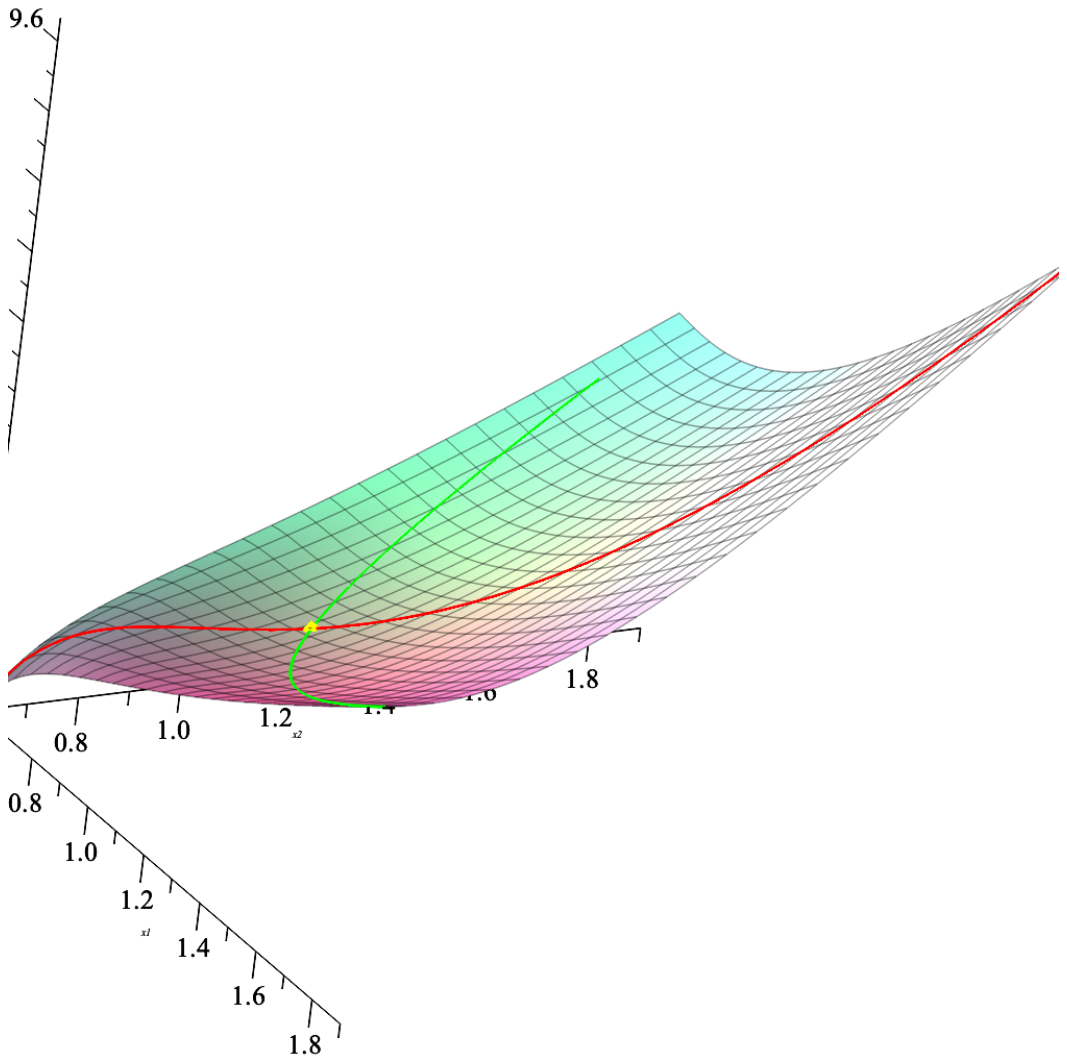}
\caption{Graph of $\scalar:\mca^{diag}_1\rightarrow\RR$ in the variables $(x_1,x_2)$ for, from left to right, $M^{13}=\SU(3)\times\SU(3)/\Delta\SO(3)$, $M^{14}=\SU(3)\times\SU(3)/\Delta T^2$ and $M^{20}=\SU(4)\times\SU(4)/\Delta\Spe(2)$, which admit two, one and none diagonal Einstein metrics (i.e., critical points, in blue), respectively.  The red curve is defined by  $x_1=x_2$ and it is invariant under the Ricci flow.  The standard metric $\gk$ ($x_1=x_2=1$) is in yellow and it is a global minimum of $\scalar: \mca^{norm}_1\rightarrow\RR$. Normal metrics give rise to the green curve.}\label{fig}
\end{figure}

\section{Diagonal Einstein metrics}\label{diag-sec} 

We study in this section the existence of $G$-invariant Einstein metrics on $M=G/\Delta K=H\times H/\Delta K$.   Recall that two Riemannian metrics are said to be {\it homothetic} if they are isometric up to scaling.  

\begin{theorem}\label{HHK-E}
On $M=H\times H/\Delta K$, the diagonal metric 
$
g=(x_1,x_2,1)_{\gk}
$ 
is Einstein if and only if $\cas_\chi=\kappa I_\qg$ for some $\kappa\in\RR$ and, 
\begin{enumerate}[{\rm (i)}] 
\item either $K$ is abelian and 
$$
x_1=x_2=\frac{\kappa+1}{2\kappa+1},
$$
\item or $K$ is semisimple, $\kil_\kg=a\kil_\hg|_\kg$ for some $a\in\RR$ and 
$$
x_1=x_2=\frac{2\kappa+1\pm\sqrt{(2\kappa+1)^2-8a(1-a+\kappa)}}{4a} =: x_\pm.
$$
If $g_\pm:=(x_\pm,x_\pm,1)_{\gk}$, then $g_+$ is not homothetic to $g_-$ (cf.\ Remark \ref{nonh}), 
$$
\ricci(g_\pm)=\rho^\pm g_\pm, \qquad\rho^\pm=\frac{(2\kappa+1)x_\pm - \kappa}{4(x_\pm)^{2}},
$$ 
and their normalized scalar curvature 
$\scalar_N(g_\pm) := \scalar(g_\pm)(\det_{\gk}{g_\pm})^{\frac{1}{\dim{M}}}$ 
are given by
$$
\scalar_N(g_\pm) = 
\frac{(2n+d)((2\kappa+1)x_\pm - \kappa)}{4(x_\pm)^{2\alpha}},  \qquad \alpha:=\frac{n+d}{2n+d}.
$$
\end{enumerate}
\end{theorem}

\begin{remark}\label{E2z-rem}
We therefore obtain that if $K$ is semisimple, then there exist exactly two diagonal Einstein metrics if and only if  
\begin{equation}\label{cond2} 
(2\kappa+1)^2 > 8a(1-a+\kappa),
\end{equation} 
there is only one (i.e., $x_+=x_-$) if equality holds and none otherwise (see Figure \ref{fig}).  In the last column of Tables \ref{table1}-\ref{table2-ss} given in Appendix \ref{table-sec} is indicated whether this condition holds or not for each suitable homogeneous space $H/K$.  It turns out that among the $17$ families and $50$ individual examples of possibilities, only a few of them do not satisfy the existence condition \eqref{cond2}, providing a great amount of homogeneous Einstein metrics.  In the tables, the check mark $\checkmark$ means that \eqref{cond2} holds and `No' means that the opposite inequality holds; equality holds in only two cases: $\Spe(24)/\Spe(8)^3$ and $\Spe(12)/\Spe(3)^4$ (see Table \ref{table2-ss}).  
\end{remark}

\begin{remark}\label{E2z-rem8}
In the case when $H/K$ is isotropy irreducible and $K$ is simple (i.e., $\pg_1$, $\pg_2$ and $\pg_3$ are $\Ad(K)$-irreducible), it follows from \eqref{pipj} that the only proper intermediate Lie subalgebras between $\Delta\kg$ and $\ggo=\hg\oplus\hg$ are 
$$
\Delta\hg \quad\mbox{and}\quad \kg\oplus\kg=\Delta\kg\oplus\pg_3\subset \Delta\kg\oplus\pg_1\oplus\pg_3, \Delta\kg\oplus\pg_2\oplus\pg_3.
$$
This implies that the graph attached to $G/\Delta K$ as in \cite{BhmWngZll} whose vertices are (connected components of) intermediate subalgebras and whose edges correspond to inclusions between two subalgebras is disconnected.  The existence of a non-diagonal $G$-invariant Einstein metric on these spaces for which condition \eqref{cond2} does not hold is therefore predicted by the Graph Theorem (see \cite[Theorem C]{BhmWngZll}).  This will be studied in \S\ref{HHKsym-sec}.  
\end{remark}

\begin{remark}\label{E2z-rem3}
Since $\kappa=\tfrac{d(1-a)}{n}$ by \eqref{kap}, where $d=\dim{K}$ and $n=\dim{H/K}$, the solutions $x_\pm$ only depend on $a$, $d$ and $n$ (or $a$, $\dim{H}$ and $\dim{K}$).  Recall that $0<\kappa\leq\unm$ and $0<a<1$.  
\end{remark}

\begin{remark}\label{E2z-rem6}
It can also be checked in the tables that $a<\kappa$ for most of the spaces satisfying the existence condition \eqref{cond2}, with the only exceptions of $\SO((m-1)(2m+1))/\Spe(m)$ and $\Spe(mk)/\Spe(k)^m$ in Tables \ref{table3-fam2} and \ref{table2-ss}, respectively.  On the other hand, the inequality
\begin{equation}\label{cond4} 
a+\kappa<1 \qquad (\mbox{if and only if}\quad d<n)
\end{equation} 
always holds for such spaces since either $a<\kappa\leq \unm$ or it can be easily checked that it holds for the two exceptions.  
\end{remark}

\begin{remark}\label{E2z-rem5}
Using that $\kappa=\tfrac{d(1-a)}{n}$, it is straightforward to show that condition \eqref{cond2} becomes 
\begin{equation}\label{cond3} 
q(a):=4(2n^2+2nd+d^2)a^2 - 4(2n^2+3nd+2d^2)a + (n+2d)^2 > 0,
\end{equation} 
which is equivalent to $0<a<\alpha_1$ or $\alpha_2<a<1$, where $0<\alpha_1<\alpha_2<1$ are the roots of the quadratic polynomial $q$.  But since $a< \tfrac{n+2d}{2n+2d}$ (see \cite[Theorem 11, pp.35]{DtrZll}) and
$q(\tfrac{n+2d}{2n+2d})= -\frac{n^3 (n+2d)}{(n+d)^2}<0$, we obtain that $0<a<\alpha_1$ always holds if $a$ satisfies condition \eqref{cond3}.  Unfortunately, the number $\alpha_1(n,d)$ has a very long expression.      
\end{remark}

\begin{remark}\label{E2z-rem7}
Since $x_\pm$ are the roots of $f(x)=2ax^2-(2\kappa+1)x+1-a+\kappa$ and $f(1)=a-\kappa$ (assume that $x_-<x_+$), we obtain that 
$$
x_- < \tfrac{2\kappa+1}{4a} < x_+ \le \tfrac{2\kappa+1}{2a} \qquad 
\mbox{and} \qquad  x_-<1<x_+, \quad\mbox{if}\quad a<\kappa, 
$$
see Remark \ref{E2z-rem6}.  For the two exceptions with $a>\kappa$, it is easy to check that  $x_-<x_+<1$ for $\SO((m-1)(2m+1))/\Spe(m)$ and $1<x_-<x_+$ for $\Spe(mk)/\Spe(k)^m$.  On the other hand, by the existence condition \eqref{cond2}, 
\begin{align*} 
f\left(\tfrac{2(1-a+\kappa)}{2\kappa+1}\right) 
 = 2a \tfrac{4(1-a+\kappa)^2}{ (2\kappa+1)^2 }  - (1-a+ \kappa) 
 =  \tfrac{(1-a+ \kappa)}{ (2\kappa+1)^2 } \left(  8a (1-a+\kappa) - (2\kappa+1)^2 \right)  < 0,
\end{align*}
hence 
\begin{equation}\label{xmxm}
x_- < \frac{2(1-a+\kappa)}{2\kappa+1}  <  x_+.
\end{equation}
\end{remark}

\begin{remark}\label{E2z-rem2}
In the case when $H/K$ is a symmetric space we have that $\kappa=\unm$ and so  
$$
x_\pm
=\tfrac{1\pm\sqrt{1-a\left(3-2a\right)}}{2a},
$$
which implies that the existence of two diagonal Einstein metrics is equivalent to $a<\unm$ (see Table \ref{table1}).   
\end{remark}

\begin{remark}\label{E2z-rem4}
The Einstein metrics obtained are a positive multiple of the metric $g=(1,1,x_3)_{\gk}$, where 
$$
x_3:=\tfrac{2\kappa+1\pm \sqrt{(2\kappa+1)^2 - 8a(1-a+\kappa)}}{2(1-a+\kappa)}, 
$$
which is giving by $x_3=\tfrac{2\kappa+1}{\kappa+1}$ in the abelian case and by 
$$
x_3=\tfrac{2\left(1\pm \sqrt{1 - a(3-2a)}\right)}{3-2a},
$$
in the case when $H/K$ is a symmetric space. 
\end{remark}

\begin{proof}[Proof of Theorem \ref{HHK-E}]
Assume that $g$ is Einstein.  It follows from Proposition \ref{ricHHK}, (i) and (ii) that $\cas_{\chi}=\kappa I_{\qg}$ for some $\kappa>0$.  Moreover, we obtain the following formulas for the Ricci eigenvalues $r_1,r_2,r_{3,0},\dots,r_{3,t}$ of $g$ on $\pg_1,\pg_2,\pg_3^0,\dots,\pg_3^t$, respectively: 
$$
r_1
= \tfrac{1+2\kappa}{4}\tfrac{1}{x_1} - \tfrac{\kappa}{4}\tfrac{1}{x_1^2}, \qquad 
r_2 
= \tfrac{1+2\kappa}{4}\tfrac{1}{x_2} - \tfrac{\kappa}{4}\tfrac{1}{x_2^2}, \qquad
r_{3,l} := a_l\left(\tfrac{1}{2}-\tfrac{1}{8}\left(\tfrac{1}{x_1^2} +\tfrac{1}{x_2^2}\right)  \right)  
+\tfrac{1}{8}\left(\tfrac{1}{x_1^2} + \tfrac{1}{x_2^2}\right).
$$
The factor multiplying $a_l$ in the formula for $r_{3,l}$ vanishes if and only if 
\begin{equation}\label{par}
\tfrac{1}{x_1^2} +\tfrac{1}{x_2^2}=4,
\end{equation} 
in which case equation $r_2=r_{3,l}$ is equivalent to 
$$
\tfrac{(1+2\kappa) x_2 - \kappa}{4x_2^2}   = \unm.
$$
This implies that either $x_2=\tfrac{1}{2}$ or $x_2 = \kappa\leq\unm$, which contradicts \eqref{par} since both alternatives give that $\tfrac{1}{x_1^2} +\tfrac{1}{x_2^2}>4$.  We therefore obtain that either $K$ is abelian (i.e., $\kg=\kg_0$) or $K$ is semisimple and $a_1=\dots=a_t=:a$.    

On the other hand, since $r_1=r_2$ if and only if 
$$
\tfrac{1+2\kappa}{4}\left(\tfrac{1}{x_1} - \tfrac{1}{x_2}\right)  =  \tfrac{\kappa}{4}\left(\tfrac{1}{x_1^2}-\tfrac{1}{x_2^2}\right), 
$$
we obtain that either $x_1=x_2$ or 
$$
\tfrac{1}{x_1}=\tfrac{1+2\kappa}{\kappa} - \tfrac{1}{x_2} = \tfrac{(1+2\kappa)x_2-\kappa}{\kappa x_2}. 
$$
If $x_1=x_2$, then $r_2=r_3$ if and only if
$$
 \kappa+1-a - (1+2\kappa)x_1+ 2ax_1^2 = 0,
$$
whose roots are $x^\pm$.  Otherwise, if $x_1=\tfrac{\kappa x_2}{(1+2\kappa)x_2-\kappa}$, then it is straightforward to see that $r_2=r_3$ if and only if $x_2$ is a root of the quadratic polynomial $d x^2 + b x+ c$, where
$$
d:=a(4 \kappa +1) - (2\kappa+1)^2, \quad 
b := 2 \kappa (2\kappa+1) (\kappa-a+1), \quad
c:= - 2 \kappa^2 (\kappa-a+1).
$$ 
Since 
$$ 
b^2-4dc = 4\kappa^2(\kappa - a + 1) (4\kappa^3 - 4a\kappa^2 + 4a\kappa - 3\kappa + a - 1)
$$
and 
$$ 
4\kappa^3 - 4a\kappa^2 + 4a\kappa - 3\kappa + a - 1
= 2 \kappa (2 \kappa^2  - 1) + 4a\kappa (  \kappa -1)  - (\kappa - a + 1) <0,
$$
we conclude that no solutions arise in this case.   

Finally, we compute the normalized scalar curvature: if $x=x^\pm$, then 
\begin{align*}
\scalar_N(g_\pm) =& 
\dim{M}\, r_1\,x^{\tfrac{2\dim{\qg}}{\dim{M}}}
=(2n+d)\left(\tfrac{1+2\kappa}{4}\tfrac{1}{x} - \tfrac{\kappa}{4}\tfrac{1}{x^2}\right) x^{\tfrac{2n}{2n+d}} \\ 
=& \tfrac{2n+d}{4}((1+2\kappa)x - \kappa) x^{\tfrac{2n}{2n+d}-2} 
=\tfrac{2n+d}{4}((1+2\kappa)x - \kappa)x^{-2\alpha}.   
\end{align*}
In order to show that $\scalar_N(g_+)\ne\scalar_N(g_-)$, which implies that $g_+$ and $g_-$ are not homothetic, we consider the function 
$$
s(x):=\scalar_N((x,x,1)_{\gk}) =\unc(2adx^2 + (4d(1-a) +2n) x-(1-a)d) x^{-2\alpha}, \qquad x>0, 
$$
whose derivative is given by 
$$
s'(x)=\tfrac{2anx^2 + (2d(1-a) - n) x-(1-a)(n+d)}{2(2n+d) x} d x^{-2\alpha}.  
$$
Thus $s'$ only vanishes at $x_+$ and $x_-$ and so $s(x_+)\ne s(x_-)$, concluding the proof. 
\end{proof}

\subsection{Canonical variations viewpoint}\label{canvar}  
We prove in this section that the Einstein metrics obtained in Theorem \ref{HHK-E} are actually canonical variations of the submersion $H\times H/\Delta K\rightarrow H/K\times H/K$ with fiber the symmetric space $K\times K/\Delta K$ (i.e., the red curve in Figure \ref{fig}).   

In \cite[9.90]{Bss}, starting from three compact Lie groups $K\subset H\subset G$, the Riemannian submersion $\pi:M\rightarrow B$ with totally geodesic fibers defined by  
$$
(F=H/K,\hat{g}) \longrightarrow (M=G/K,\gk)  \longrightarrow (B=G/H,\check{g}),
$$  
is used to obtain $G$-invariant Einstein  metrics on $M$, where $\hat{g}$, $\gk$ and $\check{g}$ are all determined by $-\kil_\ggo$.  Assume that $\hat{g}$ and $\check{g}$ are Einstein metrics with Einstein constants $\hat{\lambda}$ and $\check{\lambda}$, respectively.  It is proved that the following condition, 
\begin{equation}\label{Besse}
\check{\lambda}^2\geq (1-c)\left(1+2\tfrac{\dim{F}}{\dim{B}}\right)\hat{\lambda},
\end{equation}
implies the existence of exactly two (if $>$ holds) or one (if $=$ holds) Einstein metrics on $M=G/K$, which are canonical variations of the submersion (see \cite[9.67]{Bss}).  

It is easy to check that the Einstein metrics on $M=H\times H/\Delta K$ given in Theorem \ref{HHK-E} (see Remark \ref{E2z-rem4}) can be obtained in the above way by using the submersion attached to $\Delta K\subset K\times K\subset H\times H$, given by 
$$
(F=K\times K/\Delta K,\tfrac{1}{a}\gk) \longrightarrow (M=H\times H/\Delta K,\gk)  \longrightarrow (B=H/K\times H/K,\gk+\gk),
$$  
and the following notation correspondence:
$$
\begin{array}{c|cccccccccc}
{\rm Besse} &G&H&K&\hat{\lambda}&\check{\lambda}&c&a&\dim{F}&\dim{B}&\dim{H} \\ \hline 
{\rm Here} &H\times H&K\times K&\Delta K&\tfrac{a}{2}&\rho&a&\kappa&d&2n&2d
\end{array}.
$$
Indeed, condition \eqref{Besse} becomes 
$$
\tfrac{(2\kappa+1)^2}{16} \geq (1-a)(1+2\tfrac{d}{2n})\tfrac{a}{2} = \tfrac{a}{2}(1-a+(1-a)\tfrac{d}{n}) = \tfrac{a}{2}(1-a+\kappa),  
$$
which coincides exactly with our condition \eqref{cond2}.  

We conclude that the large class of Einstein metrics provided by Theorem \ref{HHK-E} was predicted in Besse's book \cite{Bss} more than $35$ years ago.

\section{Stability}\label{stab-sec}   

In this section, we use \cite{stab-dos} to study the stability of the Einstein metrics on $M=H\times H/\Delta K$ given in Theorem \ref{HHK-E}.  

In \cite{Hlb}, Hilbert  proved that Einstein metrics on a compact manifold $M$ are precisely the critical points of the total scalar curvature functional, 
\begin{equation}\label{sct}
\widetilde{\scalar}:\mca_1\rightarrow\RR, \qquad \widetilde{\scalar}(g):=\int_M \scalar(g)\; d\vol_g,
\end{equation}
where $\mca_1$ is the space of all unit volume Riemannian metrics on $M$ (see \cite[4.21]{Bss}).  Remarkably,   
if we consider the restriction of $\widetilde{\scalar}$ to the space $\cca_1$ of all unit volume constant scalar curvature metrics, then the coindex and nullity of any Einstein metric are both finite modulo trivial variations (see \cite{Brg, Kso2}).  A fundamental problem is therefore to determine whether a given Einstein metric $g$ is {\it stable} (or linearly stable), i.e., coindex and nullity indeed vanish, giving rise to an isolated (up to homothety) local maximum of $\widetilde{\scalar}|_{\cca_1}$ (see \cite{stab, Sch}).  

The tangent space $T_g\cca_1$ coincides, modulo trivial variations, with $\tca\tca_g=\Ker\delta_g\cap\Ker\tr_g$, the space of divergence-free (or transversal) and traceless symmetric $2$-tensors, so-called {\it TT-tensors}, and if $\ricci(g)=\rho g$, then for any $T\in\tca\tca_g$ the Hessian is given by, 
\begin{equation}\label{LL-intro}
\widetilde{\scalar}''_g(T,T) = \unm\la(2\rho\id-\Delta_L)T,T\ra_g,
\end{equation}
where $\Delta_L$ is the {\it Lichnerowicz Laplacian} of $g$ (see \cite[4.64]{Bss}).  Thus $g$ is {\it stable} (resp.\ {\it unstable} ) if and only if $2\rho<\lambda_L$ (resp.\ $\lambda_L<2\rho$), where $\lambda_L$ denotes the smallest eigenvalue of $\Delta_L|_{\tca\tca_g}$.  The number $\lambda_L$ is very hard to compute or even estimate, it is known for only few metrics.    

If $M=G/K$ is a compact homogeneous space, then $G$-invariant Einstein metrics are precisely the critical points of the scalar curvature functional,
$$
\scalar:\mca_1^G\longrightarrow \RR,  
$$ 
where $\mca^G_1$ is the finite-dimensional manifold of all unit volume $G$-invariant metrics on $M$.  It is proved in \cite[\S 3.4]{stab-tres} that 
$$
T_g\mca_1^G = T_gN_G(K)^*g \oplus \tca\tca_g^G, 
$$
where $N_G(K)\subset\Diff(M)$ is the normalizer of $K$ and $\tca\tca_g^G$ is the space of $G$-invariant TT-tensors.  In this light, an Einstein metric $g\in\mca_1^G$ is called  $G$-{\it stable} when $\scalar''_g|_{\tca\tca_g^G}<0$, or equivalently, $2\rho <\lambda_L^G$, where $\lambda_L^G$ is the smallest eigenvalue of $\Delta_L|_{\tca\tca_g^G}$.  Note that $\lambda_L\leq\lambda_L^G$.  In the case when $g$ is $G$-{\it unstable}, i.e., $\lambda_L^G<2\rho$ (or $\scalar''_g(T,T)>0$ for some $T\in{\tca\tca_g^G}$), one has:
\begin{enumerate}[{\small $\bullet$}]
\item $g$ is indeed unstable as a critical point of $\widetilde{\scalar}:\cca_1\rightarrow\RR$. 

\item $g$ is also dynamically unstable for the normalized Ricci flow, in the sense that there is an ancient solution emerging from $g$ (see \cite{CaoHe, Krn}).  

\item the metric $g$ does not realize the Yamabe invariant of $M$ (see \cite[Theorem 5.1]{BhmWngZll}).   
\end{enumerate}

For any reductive decomposition $\ggo=\kg\oplus\pg$ of $M=G/K$, the Lichnerowicz Laplacian $\Delta_L$ restricted to $\tca\tca_g^G$ is encoded in the self-adjoint operator defined by 
$$
\lic_\pg=\lic_\pg(g):\sym(\pg)^K\longrightarrow\sym(\pg)^K, \qquad  
\Delta_LT = \la\lic_\pg A\cdot,\cdot\ra_g, \qquad\forall T=g(A\cdot,\cdot) \in\tca\tca_g^G,   
$$ 
where $\sym(\pg)^K$ is the space of all $g$-symmetric and $\Ad(K)$-invariant maps of $\pg$ and $\la A,B\ra_g:=\tr{AB}$.  Thus $\lambda_L^G$ is the smallest eigenvalue of the operator $\lic_\pg|_{\tca\tca_g^G}$.  If $\pg=\pg_1\oplus\dots\oplus\pg_r$ is a $Q$-orthogonal decomposition 
 in $\Ad(K)$-invariant subspaces, where $Q$ is any bi-invariant inner product on $\ggo$, then we consider the subspace   
$$
S:=\spann\left\{ \tfrac{1}{\sqrt{n_1}}I_1,\dots, \tfrac{1}{\sqrt{n_r}}I_r\right\} \subset \sym(\pg)^K, \qquad n_k:=\dim{\pg_k}, \qquad I_k|_{\pg_i}:=\delta_{ki}I|_{\pg_i}.
$$
Let $L=L(g)$ denote the $r\times r$-matrix of $\lic_\pg$ restricted and projected on $S$ with respect to the above orthonormal basis of $S$, i.e., 
$$
L_{km}:=\la\lic_\pg \tfrac{1}{\sqrt{n_k}}I_k, \tfrac{1}{\sqrt{n_m}}I_m\ra_g.
$$  
Note that if we set 
$$
\mca^{diag}:=\{ y_1Q|_{\pg_1}+\dots+y_rQ|_{\pg_r}: y_i>0\}\subset\mca^G,
$$
then $S=T_g\mca^{diag}$ and so $2\rho I-L$ is the Hessian of $\scalar:\mca^{diag}\rightarrow\RR$ at $g$.  In particular,
\begin{equation}\label{specL} 
\Spec(L|_{S\cap \tca\tca_g^G})\subset [\lambda_L^G,\lambda_L^{G,max}], 
\end{equation} 
where $\lambda_L^{G,max}$ is the largest eigenvalue of the operator $\lic_\pg|_{\tca\tca_g^G}$, so the eigenvalues of $L$ may help to establish the stability type of an Einstein metric $g\in\mca^G$.   

The following formula in terms of the corresponding structural constants $[ijk]$'s (see \eqref{ijk-gen}) will be very useful.  

\begin{theorem}\label{formLp-intro} \cite[Theorem 3.1]{stab-dos}   
For any metric $g=x_1Q|_{\pg_1}+\dots+x_rQ|_{\pg_r}\in\mca^G$, 
$$
L_{kk} = \tfrac{1}{n_k}\sum_{i,j\ne k} [ijk]\tfrac{x_k}{x_ix_j}  + \tfrac{1}{n_k}\sum_{i\ne k}  [ikk]\tfrac{x_i}{x_k^2}, \quad
L_{km} = \tfrac{1}{\sqrt{n_k}\sqrt{n_m}} \sum_{i} [ikm]\tfrac{x_i^2-x_k^2-x_m^2}{x_ix_kx_m}.  
$$
\end{theorem}

In order to compute the stability type of the Einstein metrics $g_\pm$ on $M=G/\Delta K=H\times H/\Delta K$ given in Theorem \ref{HHK-E}, we consider the reductive complement $\pg=\pg_1\oplus\pg_2\oplus\pg_3$ defined in \eqref{gBdec}.  Thus $\dim{\mca_1^{diag}}=2$.  We set $g_+=g_-$ and $x_+=x_-=\frac{\kappa+1}{2\kappa+1}$ in the case when $K$ is abelian.  

It is easy to prove that 
$$
N_\ggo(\Delta\kg)=\Delta\kg+(\kg_0+\qg_0,\kg_0+\qg_0), \qquad \qg_0:=\{ X\in\qg:[\kg,X]=0\}.
$$
Since $\cas_\chi=\kappa I_\qg$ is a necessary condition for the existence of $g_\pm$, we have that $q_0=0$, and using in addition that $\ad{(Z_1,Z_2)}|_\pg\in\sog(\pg,g)$ for all $Z_1,Z_2\in\kg_0$, we obtain that  $T_{g_\pm}N_G(\Delta K)^*g_\pm =0$.  Thus 
$$
T_{g_\pm}\mca_1^G = \tca\tca_{g_\pm}^G =\{ g(A\cdot,\cdot): A\in\sym(\pg)^K,\; \tr{A}=0\}, 
$$
and written in terms of $\left\{  \tfrac{1}{\sqrt{n}}I_1, \tfrac{1}{\sqrt{n}}I_2, \tfrac{1}{\sqrt{d}}I_3\right\}$, we have that $S\cap \tca\tca_{g_\pm}^G=\{ (\sqrt{n},\sqrt{n},\sqrt{d})\}^\perp$.  

\begin{lemma}\label{Lp} 
For the Einstein metrics $g_\pm=(x_\pm,x_\pm,1)_{\gk}$, we have that 
$$
L(g_\pm) = \frac{\kappa}{2 x_\pm^2}  
\left[\begin{matrix} 1&0&-\sqrt{\frac{n}{d}}\\ 0&1&-\sqrt{\frac{n}{d}} \\ 
-\sqrt{\frac{n}{d}} &-\sqrt{\frac{n}{d}}& \frac{2n}{d}
\end{matrix}\right],
$$ 
with eigenvalues given by 
$$
0, \qquad \lambda_1^\pm=  \frac{\kappa}{2 (x_\pm)^2}, \qquad \lambda_2^\pm = \frac{2n+d}{d}\frac{\kappa}{2(x_\pm)^2},
$$
and respective eigenvectors, 
$$
(\sqrt{n},\sqrt{n},\sqrt{d}), \qquad (1,-1,0), \qquad (-\sqrt{d},-\sqrt{d},2\sqrt{n}).
$$
\end{lemma}

\begin{proof}
Using Theorem \ref{formLp-intro} and the structural constants of $\pg=\pg_1\oplus\pg_2\oplus\pg_3$ given in Lemma \ref{ijk}, we obtain that
$$
L_{ii}=\tfrac{1}{n}[3ii]\tfrac{1}{(x_\pm)^2}=\tfrac{\kappa}{2(x_\pm)^2}, \quad i=1,2, \qquad
L_{33}=\tfrac{1}{d}[113]\tfrac{1}{(x_\pm)^2}+\tfrac{1}{d}[223]\tfrac{1}{(x_\pm)^2}=\tfrac{\kappa n}{(x_\pm)^2d}, \quad
$$
and
$$
L_{12}=0, \qquad L_{i3}=\tfrac{1}{\sqrt{n}\sqrt{d}}[ii3]\tfrac{-1}{(x_\pm)^2}=\tfrac{-\kappa n}{\sqrt{n}\sqrt{d}2(x_\pm)^2}, \quad i=1,2,
$$
concluding the proof.  
\end{proof}

It follows from Lemma \ref{Lp} that 
\begin{equation}\label{lamb}
\lambda_L^G(g_\pm)\leq\lambda_1^\pm<\lambda_2^\pm\leq\lambda_L^{G,max}(g_\pm),
\end{equation}
which allows to establish the instability of the Einstein metrics $g_\pm$ as follows.  Recall from Theorem \ref{HHK-E} that $\ricci(g_\pm)=\rho^\pm g_\pm$, where $\rho^\pm:=\tfrac{(2\kappa+1) x_\pm -\kappa}{4(x_\pm^2)}$.

\begin{proposition}\label{Lpe}
The following inequalities hold: if $K$ is abelian then $\lambda_1<2\rho<\lambda_2$, and for $K$ semisimple,
$$
\lambda_1^-<2\rho^-<\lambda_2^-, \qquad 
\lambda_1^+<\lambda_2^+<2\rho^+.
$$
In particular, 
\begin{enumerate}[{\rm (i)}] 
\item $g_-$ is $G$-unstable with coindex $\geq 1$ and a saddle point of $\scalar:\mca_1^{diag}\rightarrow\RR$.  Moreover, the function $\scalar(g_t)$, where $g_t:=(tx^-,\frac{1}{t}x^-,1)\in\mca_1^{diag}$, attains a local minimun at $t=1$ (i.e., at $g_1=g_-$).  All this also holds when $K$ is abelian for $g_+=g_-$.  

\item $g_+$ is $G$-unstable with coindex $\geq 2$ and a local minimum of $\scalar:\mca_1^{diag}\rightarrow\RR$. 
\end{enumerate}
\end{proposition}

\begin{remark}
The possible situations are illustrated in Figure \ref{fig}.  
\end{remark}

\begin{remark}\label{nonh}
Since 
$$
N_G(\Delta K)=\{ u=(h_1,h_2)\in G:h_1,h_2\in N_H(K), \; h_2^{-1}h_1\in C_H(K)\},
$$
we obtain that $dI_u|_o=\Ad(u)|_\pg$ leaves each $\pg_i$ invariant for $i=1,2,3$ and for any $u\in N_G(\Delta K)$, so the pull back action of $N_G(\Delta K)$ on $\mca^G$ leaves the submanifold $\mca^{diag}$ invariant.  This implies that $g_+$ is not $G$-equivariantly isometric up to scaling to $g_-$ (see \cite[Section 3.3]{stab-tres}); indeed, their stability types as critical points of $\scalar:\mca_1^{diag}\rightarrow\RR$ are different by the above proposition.  
\end{remark} 

\begin{remark}\label{stab-inv}
It is worth pointing out the above application of stability to differentiate two homogeneous Einstein metrics up to homothecy.  The only other invariant that has been used in the literature so far is the normalized scalar curvature, which is often only numerically approximated and sometimes the corresponding values are suspiciously too close (see e.g.\ Tables \ref{Sc1} and \ref{Sc2}).   
\end{remark} 

\begin{proof}
If $K$ is abelian, then $2\rho=\frac{(2\kappa+1)^2}{2(\kappa+1)^2}$ and by Lemma \ref{Lp}, $\lambda_1=\frac{(2\kappa+1)^2\kappa}{2(\kappa+1)^2}$ and $\lambda_2=\frac{2n+d}{d}\frac{(2\kappa+1)^2\kappa}{2(\kappa+1)^2}$.  Using that $\kappa=\tfrac{d}{n}$ we obtain that $\lambda_1<2\rho<\lambda_2$.   

In the case when $K$ is semisimple, it follows from Lemma \ref{Lp} and the fact that $x_-\leq x_+$ that $\lambda_1^\pm<2\rho^\pm$ if and only if 
$ \tfrac{2\kappa}{ 2\kappa+1} <  x_-$.  By Theorem \ref{HHK-E}, (ii), this is equivalent to 
$$ 
\tfrac{8 \kappa a}{2\kappa+1} <  2\kappa+1 - \sqrt{(2\kappa+1)^2 -8a (\kappa+1-a)},
$$
or to
$$
8a (\kappa+1-a) -  16 \kappa a  +  \tfrac{(8\kappa a)^2}{(2\kappa+1 )^2}  >  0,   
$$
which holds by \eqref{cond4}.  On the other hand,  $2\rho^-<\lambda_2^-$ if and only if $d((2\kappa+1)x_- -2\kappa)<2n\kappa$, which by \eqref{kap} is equivalent to 
$x_- < \frac{2( \kappa +1-a)}{1+2\kappa}$, and analogously, $\lambda_2^+  <  2\rho^+$ is equivalent to 
$ \frac{2( \kappa +1-a)}{1+2\kappa}<  x_+$, so they both follow from \eqref{xmxm}.    

Finally, parts (i) and (ii) follow from the fact that $2\rho^\pm-L(g_\pm)$ is the Hessian at $g_\pm$ of the function $\scalar:\mca^{diag}\rightarrow\RR$ and $\la LA,A\ra_{g_+}=\la L_\pg A,A\ra_{g_+}$ for all $A\in S\cap\tca\tca_{g_+}^G$, concluding the proof.   
\end{proof}

\section{Non-diagonal Einstein metrics}\label{HHKsym-sec}

The $\gk$-orthogonal $\Ad(K)$-invariant decomposition $\pg=\pg_1\oplus\pg_2\oplus\pg_3$ (see \eqref{gBdec}) satisfies that $\pg_1\simeq\pg_2$ as $\Ad(K)$-representations, giving rise to the following {\it non-diagonal} $G$-invariant metrics on the homogeneous space $M=G/\Delta K=H\times H/\Delta K$.  In this section, we will explore for Einstein metrics on this class.  

We consider a $G$-invariant metric 
$$
g=\left(x_1,x_2,x_3,x_4\right)
$$ 
defined by $g(\pg_1,\pg_3)=g(\pg_2,\pg_3)=0$ and for all $X_i,Y_i\in\qg$, $Z,W\in\kg$, 
\begin{align}
g((X_1,X_2),(Y_1,Y_2)):=& x_1(-\kil_\hg)(X_1,Y_1) +x_2(-\kil_\hg)(X_2,Y_2) \notag \\ 
&+x_4\left((-\kil_\hg)(X_1,Y_2)+(-\kil_\hg)(X_2,Y_1)\right),  \label{defg4} \\ 
g((Z,-Z),(W,-W)):=& x_3(-\kil_\ggo)((Z,-Z),(W,-W)) = 2x_3(-\kil_\hg)(Z,W), \notag
\end{align} 
where $x_1,x_2,x_3>0$ and $x_1x_2>x_4^2$.  The matrix of $g$ with respect to $\gk$ is therefore given by 
\begin{equation}\label{mat}
\left[\begin{matrix}
x_1I_{n}&x_4I_{n}&0\\ 
x_4I_{n}&x_2I_{n}&0\\ 
0&0&x_3I_{d}
\end{matrix}\right],
\end{equation}
where $d:=\dim{K}$ and $n:=\dim{H/K}$.  

\begin{remark}\label{all}
If the isotropy representation of $H/K$ is irreducible and of real type and $K$ is simple, then any $H\times H$-invariant metric on the homogeneous space $M=H\times H/\Delta K$ is of this form.  
\end{remark}

A $g$-orthogonal $\Ad(K)$-invariant decomposition of $\pg$ is given by 
$$
\pg=\qg_1\oplus\qg_2\oplus\qg_3,
$$
where 
$$
\begin{array}{c}
\qg_1:=\pg_1=\left\{(X,0):X\in\qg\right\}, \qquad \qg_2:=\left\{(-x_4X,x_1X):X\in\qg\right\}, \\ \\ \qg_3:=\pg_3=\left\{(Z,-Z):Z\in\kg\right\}. 
\end{array}
$$  
If $\{ Z_\alpha\}$ is a $(-\kil_\hg)$-orthonormal basis of $\kg$ and $\{ e_\alpha\}$ is a $(-\kil_\hg)$-orthonormal basis of $\qg$, then it is easy to check that
$$
\left\{Y^1_\alpha:=\tfrac{1}{\sqrt{x_1}}(e_\alpha,0)\right\}, \quad 
\left\{Y^2_\alpha:=\tfrac{1}{\sqrt{x_1(x_1x_2-x_4^2)}}(-x_4e_\alpha,x_1e_\alpha)\right\}, \quad 
\left\{ Y^3_\alpha:=\tfrac{1}{\sqrt{2x_3}}(Z_\alpha,-Z_\alpha)\right\},
$$ 
are $g$-orthonormal bases of $\qg_1$, $\qg_2$ and $\qg_3$, respectively.  We assume that the basis $\{ Z_\alpha\}$ is adapted to the decomposition \eqref{decs}, so the basis $\{ Y^3_\alpha\}$ is adapted to the decomposition $\pg_3=\pg_3^0\oplus\dots\oplus\pg_3^t$.  

An explicit description of the Ricci tensor for these non-diagonal metrics involves heavy computations, even in the case when $H/K$ is an irreducible symmetric space.  

\begin{proposition}\label{ri}
If $H/K$ is an irreducible  symmetric space, then the Ricci tensor of the metric $g=(x_1,x_2,x_3,x_4)$ on $M^{2n+d}=H\times H/\Delta K$ is given for any 
$$
\overline{X}=(X,0)\in\pg_1, \quad X\in\qg, \quad \overline{Y}=(0,Y)\in\pg_2, \quad Y\in\qg, \quad \overline{Z}=(Z,-Z)\in\pg_3^l, \quad Z\in\kg_l,
$$
$l=0,\dots,t$, as follows: $\quad\ricci(g)(\overline{X},\overline{Z}) = 0$, $\quad \ricci(g)(\overline{Y},\overline{Z}) = 0$,
\begin{align*}
\ricci(g)(\overline{X},\overline{X}) =& \left(\frac{x_3}{8x_1} + \frac{x_3x_4^2}{8x_1(x_1x_2-x_4^2)} 
-\frac{x_1x_4^2}{2(x_1x_2-x_4^2)x_3} - \frac{1}{2}\right) \kil_\hg(X,X),\\   
\ricci(g)(\overline{Y},\overline{Y}) =& \left(\frac{x_1x_3}{8(x_1x_2-x_4^2)} 
+\frac{x_1x_2-x_4^2}{8x_1x_3}
-\frac{(x_1x_2+x_4^2)^2}{8x_1(x_1x_2-x_4^2)x_3} -  \frac{1}{2}\right) \kil_\hg(Y,Y),\\  
\ricci(g)(\overline{X},\overline{Y}) =& \left(\frac{x_3x_4}{8(x_1x_2-x_4^2)} 
-\frac{x_4}{4x_3}
- \frac{x_4(x_1x_2+x_4^2)}{4(x_1x_2-x_4^2)x_3}\right)  \kil_\hg(X,Y),\\ 
\ricci(g)(\overline{Z},\overline{Z}) =& \left(a_l(R-1)-R\right)
\kil_\hg(Z,Z), \qquad \kil_{\kg_l}=a_l\kil_\hg|_{\kg_l},
\end{align*}
where
\begin{equation}\label{defR}
R:=
-\frac{2x_4^2}{x_1x_2-x_4^2}
+ \frac{x_3^2}{4x_1^2} 
+\frac{x_3^2(x_1^2-x_4^2)^2}{4x_1^2(x_1x_2-x_4^2)^2} 
+\frac{x_3^2x_4^2}{2x_1^2(x_1x_2-x_4^2)}.
\end{equation}
\end{proposition}

\begin{proof} 
Since $H/K$ is a symmetric space, $[\qg,\qg]\subset\kg$, so the Lie brackets between the $\qg_i$'s and $\pg_i$'s satisfy: 
\begin{align}
& [\qg_1,\qg_1]\subset\qg_3+\Delta\kg, \qquad [\qg_2,\qg_2]\subset\qg_3+\Delta\kg, 
\qquad [\qg_1,\qg_2]\subset\qg_3+\Delta\kg, \notag \\ 
&[\qg_3,\qg_3]\subset\Delta\kg, \qquad [\qg_3,\qg_1]\subset\qg_1, \qquad [\qg_3,\qg_2]\subset\qg_1+\qg_2 =\pg_1+\pg_2, \label{LB}\\ 
&[\pg_2,\qg_1]=0, \qquad [\pg_2,\qg_2]\subset\qg_3+\Delta\kg, \qquad [\pg_2,\qg_3]\subset\qg_1+\qg_2,\; \pg_2. \notag
\end{align}
The following formula will be very useful in the proof:
\begin{equation}\label{casi}
\sum_{\substack{\alpha,\beta}} (-\kil_\hg)([X,e_\alpha],Z_\beta)(-\kil_\hg)([Y,e_\alpha],Z_\beta) = -\unm\kil_\hg(X,Y), \qquad\forall X,Y\in\qg.
\end{equation}
This follows from the fact that the left hand side is equal to both $-\tr{\ad{X}\ad{Y}}|_\kg$ and $-\tr{\ad{X}\ad{Y}}|_\qg$, or alternatively, by using that $\cas_{\chi,-\kil_\hg|_\kg}=\unm I_\qg$ (see \S\ref{Est-sec}).  On the other hand, using that $\kil_{\kg_l}=a_l\kil_\hg|_{\kg_l}$, we obtain another useful formula: 
\begin{equation}\label{casi2}
\sum_{\substack{\alpha,\beta}} (-\kil_\hg)([Z,e_\alpha],e_\beta)^2 
= -\tr{(\ad{Z}|_\qg)^2} =  \kil_\kg(Z,Z) - \kil_\hg(Z,Z) 
= (a_l-1)\kil_\hg(Z,Z), 
\end{equation}
for any $Z\in\kg_l$ (recall that $a_0=0$).  

We now start computing the Ricci tensor.  If $\overline{X}=(X,0)\in\qg_1$, $X\in\qg$, then by \eqref{Rc} and \eqref{LB},  
\begin{align*}
\ricci(g)(\overline{X},\overline{X}) 
=& -\unm\sum_{\substack{\alpha,\beta}} g([\overline{X},Y^1_\alpha]_\pg,Y^3_\beta)^2 
 -\unm\sum_{\substack{\alpha,\beta}} g([\overline{X},Y^2_\alpha]_\pg,Y^3_\beta)^2 
 -\unm\sum_{\substack{\alpha,\beta}} g([\overline{X},Y^3_\beta]_\pg,Y^1_\alpha)^2 \\
&+\unm\sum_{\substack{\alpha,\beta}} g([Y^1_\alpha,Y^3_\beta]_\pg,\overline{X})^2 
+ \unm\sum_{\substack{\alpha,\beta}} g([Y^2_\alpha,Y^3_\beta]_\pg,\overline{X})^2 
- \unm\kil_\ggo(\overline{X},\overline{X}) \\ 
=& -\unm\sum_{\substack{\alpha,\beta}} \tfrac{1}{x_12x_3}g(([(X,0),(e_\alpha,0)]_{\qg_3},(Z_\beta,-Z_\beta))^2 \\
& -\unm\sum_{\substack{\alpha,\beta}} \tfrac{1}{x_1(x_1x_2-x_4^2)2x_3} g(([(X,0),(-x_4e_\alpha,x_1e_\alpha)]_{\qg_3},(Z_\beta,-Z_\beta))^2 \\
& -\unm\sum_{\substack{\alpha,\beta}} \tfrac{1}{2x_3x_1}g(([(X,0),(Z_\beta,-Z_\beta)],(e_\alpha,0))^2 \\
& +\unm\sum_{\substack{\alpha,\beta}} \tfrac{1}{x_12x_3}g(([(e_\alpha,0),(Z_\beta,-Z_\beta)],(X,0))^2 \\
& + \unm\sum_{\substack{\alpha,\beta}} \tfrac{1}{x_1(x_1x_2-x_4^2)2x_3} g([(-x_4e_\alpha,x_1e_\alpha),(Z_\beta,-Z_\beta)],(X,0))^2 
- \unm\kil_\hg(X,X).
\end{align*}

By using \eqref{defg4} and $\kil_\ggo(\Delta\kg,\qg_3)=0$, we obtain that 
\begin{align*}
&\ricci(g)(\overline{X},\overline{X}) \\
=& -\unm\sum_{\substack{\alpha,\beta}} \tfrac{x_3}{x_12} (-\kil_\ggo)([X,e_\alpha],0),(Z_\beta,-Z_\beta))^2 \\
& -\unm\sum_{\substack{\alpha,\beta}} \tfrac{x_3}{x_1(x_1x_2-x_4^2)2} (-\kil_\ggo)((-x_4[X,e_\alpha],0),(Z_\beta,-Z_\beta))^2 \\
& -\unm\sum_{\substack{\alpha,\beta}} \tfrac{x_1}{2x_3}(-\kil_\hg)([X,Z_\beta],e_\alpha)^2 
+\unm\sum_{\substack{\alpha,\beta}} \tfrac{x_1}{2x_3}(-\kil_\hg)([e_\alpha,Z_\beta],X)^2 \\
& +  \unm\sum_{\substack{\alpha,\beta}} \tfrac{1}{x_1(x_1x_2-x_4^2)2x_3} 
g((-x_4[e_\alpha,Z_\beta],-x_1[e_\alpha,Z_\beta]),(X,0))^2 
- \unm\kil_\hg(X,X).
\end{align*}
Since the third line vanishes, it follows from \eqref{casi} that
\begin{align*}
&\ricci(g)(\overline{X},\overline{X}) \\
=& -\unm\sum_{\substack{\alpha,\beta}} \tfrac{x_3}{x_12} (-\kil_\hg)([X,e_\alpha],Z_\beta)^2 
 -\unm\sum_{\substack{\alpha,\beta}} \tfrac{x_3x_4^2}{x_1(x_1x_2-x_4^2)2} (-\kil_\hg)([X,e_\alpha],Z_\beta)^2 \\
& +  \unm\sum_{\substack{\alpha,\beta}} \tfrac{4x_1^2x_4^2}{x_1(x_1x_2-x_4^2)2x_3} 
(-\kil_\hg)([e_\alpha,Z_\beta],X)^2 
- \unm\kil_\hg(X,X)\\ 
=& \left(\tfrac{x_3}{8x_1} + \tfrac{x_3x_4^2}{8x_1(x_1x_2-x_4^2)} 
-\tfrac{x_1x_4^2}{2(x_1x_2-x_4^2)x_3} - \unm\right)\kil_\hg(X,X), 
\end{align*}
as was to be shown.  

For $\overline{Y}=(0,Y)\in\pg_2$, $Y\in\qg$, we obtain from \eqref{Rc} and \eqref{LB} that
\begin{align*}
\ricci(g)(\overline{Y},\overline{Y}) 
=&  -\unm\sum_{\substack{\alpha,\beta}} g([\overline{Y},Y^2_\alpha]_\pg,Y^3_\beta)^2 
 -\unm\sum_{\substack{\alpha,\beta}} g([\overline{Y},Y^3_\beta]_\pg,Y^2_\alpha)^2 
 -\unm\sum_{\substack{\alpha,\beta}} g([\overline{Y},Y^3_\beta]_\pg,Y^1_\alpha)^2 \\
&+\unm\sum_{\substack{\alpha,\beta}} g([Y^1_\alpha,Y^3_\beta]_\pg,\overline{Y})^2 
+ \unm\sum_{\substack{\alpha,\beta}} g([Y^2_\alpha,Y^3_\beta]_\pg,\overline{Y})^2 
- \unm\kil_\ggo(\overline{Y},\overline{Y}) \\ 
& = -\unm\sum_{\substack{\alpha,\beta}} \tfrac{1}{x_1(x_1x_2-x_4^2)2x_3} g([(0,Y),(-x_4e_\alpha,x_1e_\alpha)]_{\qg_3},(Z_\beta,-Z_\beta))^2 \\
& -\unm\sum_{\substack{\alpha,\beta}} \tfrac{1}{2x_3x_1(x_1x_2-x_4^2)} g(([(0,Y),(Z_\beta,-Z_\beta)],(-x_4e_\alpha,x_1e_\alpha))^2 \\
& -\unm\sum_{\substack{\alpha,\beta}} \tfrac{1}{2x_3x_1} g(([(0,Y),(Z_\beta,-Z_\beta)],(e_\alpha,0))^2 \\
& +\unm\sum_{\substack{\alpha,\beta}} \tfrac{1}{x_12x_3}g(([(e_\alpha,0),(Z_\beta,-Z_\beta)],(0,Y))^2 \\
& + \unm\sum_{\substack{\alpha,\beta}} \tfrac{1}{x_1(x_1x_2-x_4^2)2x_3} g([(-x_4e_\alpha,x_1e_\alpha),(Z_\beta,-Z_\beta)],(0,Y))^2 
- \unm\kil_\hg(Y,Y).  
\end{align*}
It now follows from \eqref{defg4} and $\kil_\ggo(\Delta\kg,\qg_3)=0$ that 
\begin{align*}
&\ricci(g)(\overline{Y},\overline{Y}) \\
=& -\unm\sum_{\substack{\alpha,\beta}} \tfrac{x_3}{x_1(x_1x_2-x_4^2)2} (-\kil_\ggo)((0,x_1[Y,e_\alpha]),(Z_\beta,-Z_\beta))^2 \\
& -\unm\sum_{\substack{\alpha,\beta}} \tfrac{1}{2x_3x_1(x_1x_2-x_4^2)} g((0,-[Y,Z_\beta]),(-x_4e_\alpha,x_1e_\alpha))^2 \\
& -\unm\sum_{\substack{\alpha,\beta}} \tfrac{x_4^2}{2x_3x_1} -\kil_\hg([Y,Z_\beta],e_\alpha)^2 
+\unm\sum_{\substack{\alpha,\beta}} \tfrac{x_4^2}{x_12x_3} -\kil_\hg([e_\alpha,Z_\beta],Y)^2 \\
& +  \unm\sum_{\substack{\alpha,\beta}} \tfrac{1}{x_1(x_1x_2-x_4^2)2x_3} 
g((-x_4[e_\alpha,Z_\beta],-x_1[e_\alpha,Z_\beta]),(0,Y))^2 
- \unm\kil_\hg(Y,Y),
\end{align*}
and since the fourth line vanishes, by using \eqref{casi2} we obtain the formula stated in the proposition as follows:
\begin{align*}
&\ricci(g)(\overline{Y},\overline{Y}) \\
=&  -\unm\sum_{\substack{\alpha,\beta}} \tfrac{x_3x_1^2}{x_1(x_1x_2-x_4^2)2} (-\kil_\hg)([Y,e_\alpha],Z_\beta)^2 
-\unm\sum_{\substack{\alpha,\beta}} \tfrac{(-x_2x_1+x_4^2)^2}{2x_3x_1(x_1x_2-x_4^2)} (-\kil_\hg)([Y,Z_\beta],e_\alpha)^2 \\
& +  \unm\sum_{\substack{\alpha,\beta}} \tfrac{(-x_2x_1-x_4^2)^2}{x_1(x_1x_2-x_4^2)2x_3} 
(-\kil_\hg)([e_\alpha,Z_\beta],Y)^2 
- \unm\kil_\hg(Y,Y)\\ 
=& \left(\tfrac{x_1x_3}{8(x_1x_2-x_4^2)} 
+\tfrac{x_1x_2-x_4^2}{8x_1x_3}
-\tfrac{(x_1x_2+x_4^2)^2}{8x_1(x_1x_2-x_4^2)x_3} - \unm\right)\kil_\hg(Y,Y).  
\end{align*}

We now consider $\overline{X}=(X,0)\in\qg_1$, $X\in\qg$ and $\overline{Y}=(0,Y)\in\pg_2$, $Y\in\qg$ (in particular, $\kil_\ggo(\overline{X},\overline{Y})=0$).  It follows from \eqref{LB} that  
\begin{align*}
&\ricci(g)(\overline{X},\overline{Y}) \\ 
=&-\unm\sum_{\substack{\alpha,\beta}} g([\overline{X},Y^2_\alpha]_\pg,Y^3_\beta)g([\overline{Y},Y^2_\alpha]_\pg,Y^3_\beta) 
-\unm\sum_{\substack{\alpha,\beta}} g([\overline{X},Y^3_\alpha]_{\pg},Y^1_\beta)g([\overline{Y},Y^3_\alpha]_{\pg},Y^1_\beta) \\
& +\unm\sum_{\substack{\alpha,\beta}} 
g([Y^1_\alpha,Y^3_\beta]_{\pg},\overline{X}) g([Y^1_\alpha,Y^3_\beta]_{\pg},\overline{Y})  
+\unm\sum_{\substack{\alpha,\beta}} 
g([Y^2_\alpha,Y^3_\beta]_{\pg},\overline{X}) g([Y^2_\alpha,Y^3_\beta]_{\pg},\overline{Y}),
\end{align*}
that is,
\begin{align*}
\ricci(g)(\overline{X},\overline{Y}) 
=&-\unm\sum_{\substack{\alpha,\beta}} 
\tfrac{1}{x_1(x_1x_2-x_4^2)2x_3} g([(X,0),(-x_4e_\alpha,x_1e_\alpha)]_{\qg_3},(Z_\beta,-Z_\beta)) \\ & \cdot g([(0,Y),(-x_4e_\alpha,x_1e_\alpha)]_{\qg_3},(Z_\beta,-Z_\beta)) \\ 
&-\unm\sum_{\substack{\alpha,\beta}} 
\tfrac{1}{2x_3x_1} g([(X,0),(Z_\beta,-Z_\beta)],(e_\alpha,0)) \\ & \cdot g([(0,Y),(Z_\beta,-Z_\beta)],(e_\alpha,0)) \\
&+\unm\sum_{\substack{\alpha,\beta}} 
\tfrac{1}{x_12x_3} g([(e_\alpha,0),(Z_\beta,-Z_\beta)],(X,0)) \\ & \cdot g([(e_\alpha,0),(Z_\beta,-Z_\beta)],(0,Y)) \\ 
&+\unm\sum_{\substack{\alpha,\beta}} 
\tfrac{1}{x_1(x_1x_2-x_4^2)2x_3} g([(-x_4e_\alpha,x_1e_\alpha),(Z_\beta,-Z_\beta)],(X,0)) \\ & \cdot g([(-x_4e_\alpha,x_1e_\alpha),(Z_\beta,-Z_\beta)],(0,Y)).   
\end{align*}
By using \eqref{defg4} and $\kil_\ggo(\Delta\kg,\qg_3)=0$ we obtain that 
\begin{align*}
\ricci(g)(\overline{X},\overline{Y}) 
=&-\unm\sum_{\substack{\alpha,\beta}} 
\tfrac{x_3^2}{x_1(x_1x_2-x_4^2)2x_3} (-\kil_\ggo)((-x_4[X,e_\alpha],0),(Z_\beta,-Z_\beta)) \\ & \qquad (-\kil_\ggo)((0,x_1[Y,e_\alpha]),(Z_\beta,-Z_\beta)) \\ 
&-\unm\sum_{\substack{\alpha,\beta}} 
\tfrac{1}{2x_3x_1} g(([X,Z_\beta],0),(e_\alpha,0)) g((0,-[Y,Z_\beta]),(e_\alpha,0)) \\ 
&+\unm\sum_{\substack{\alpha,\beta}} 
\tfrac{1}{x_12x_3} g(([e_\alpha,Z_\beta],0),(X,0)) g(([e_\alpha,Z_\beta],0),(0,Y)) \\ 
&+\unm\sum_{\substack{\alpha,\beta}} 
\tfrac{1}{x_1(x_1x_2-x_4^2)2x_3} g((-x_4[e_\alpha,Z_\beta],-x_1[e_\alpha,Z_\beta]),(X,0)) \\ & \qquad g((-x_4[e_\alpha,Z_\beta],-x_1[e_\alpha,Z_\beta],(0,Y)),
\end{align*}
and so by \eqref{casi},
\begin{align*}
&\ricci(g)(\overline{X},\overline{Y}) \\
=&-\unm\sum_{\substack{\alpha,\beta}} 
\tfrac{x_3(-x_4)(-x_1)}{x_1(x_1x_2-x_4^2)2} (-\kil_\hg)([X,e_\alpha],Z_\beta) (-\kil_\hg)([Y,e_\alpha],Z_\beta) \\ 
&-\unm\sum_{\substack{\alpha,\beta}} 
\tfrac{x_1(-x_4)}{2x_3x_1} (-\kil_\hg)([X,Z_\beta],e_\alpha) (-\kil_\hg)([Y,Z_\beta],e_\alpha) \\ 
&+\unm\sum_{\substack{\alpha,\beta}} 
\tfrac{x_1x_4}{x_12x_3} (-\kil_\hg)([e_\alpha,Z_\beta],X) (-\kil_\hg)([e_\alpha,Z_\beta],Y) \\ 
&+\unm\sum_{\substack{\alpha,\beta}} 
\tfrac{(x_1(-x_4)+x_4(-x_1))(x_4(-x_4)+x_2(-x_1))}{x_1(x_1x_2-x_4^2)2x_3} (-\kil_\hg)([e_\alpha,Z_\beta],X) (-\kil_\hg)([e_\alpha,Z_\beta],Y]) \\ 
=&\left(-\tfrac{x_3x_4}{4(x_1x_2-x_4^2)} 
+\tfrac{x_4}{2x_3}
+ \tfrac{x_4(x_1x_2+x_4^2)}{2(x_1x_2-x_4^2)x_3}\right) (-\unm)\kil_\hg(X,Y) \\ 
=&\left(\tfrac{x_3x_4}{8(x_1x_2-x_4^2)} 
-\tfrac{x_4}{4x_3}
- \tfrac{x_4(x_1x_2+x_4^2)}{4(x_1x_2-x_4^2)x_3}\right) \kil_\hg(X,Y), 
\end{align*}
as we had to show.  

If $\overline{Z}=(Z,-Z)\in\qg_3$, $Z\in\kg_l$, then 
\begin{align*}
\ricci(g)(\overline{Z},\overline{Z})  
=& -\unm\sum_{\substack{\alpha,\beta}} g([\overline{Z},Y^1_\alpha]_\pg,Y^1_\beta)^2 
-\unm\sum_{\substack{\alpha,\beta}} g([\overline{Z},Y^2_\alpha]_\pg,Y^1_\beta)^2 
-\unm\sum_{\substack{\alpha,\beta}} g([\overline{Z},Y^2_\alpha]_\pg,Y^2_\beta)^2  \\ 
&+\unc\sum_{\substack{\alpha,\beta}} g([Y^1_\alpha,Y^1_\beta]_\pg,\overline{Z})^2 
+\unc\sum_{\substack{\alpha,\beta}} g([Y^2_\alpha,Y^2_\beta]_\pg,\overline{Z})^2 
+\unm\sum_{\substack{\alpha,\beta}} g([Y^1_\alpha,Y^2_\beta]_\pg,\overline{Z})^2 \\ 
&- \unm\kil_\ggo(\overline{Z},\overline{Z}),
\end{align*}
or equivalently, 
\begin{align*}
\ricci(g)(\overline{Z},\overline{Z})  
=& -\unm\sum_{\substack{\alpha,\beta}} \tfrac{1}{x_1^2} g([(Z,-Z),(e_\alpha,0)],(e_\beta,0))^2 \\
& -\unm\sum_{\substack{\alpha,\beta}} \tfrac{1}{x_1^2(x_1x_2-x_4^2)} g([(Z,-Z),(-x_4e_\alpha,x_1e_\alpha)],(e_\beta,0))^2 \\ 
& -\unm\sum_{\substack{\alpha,\beta}} \tfrac{1}{x_1^2(x_1x_2-x_4^2)^2} g([(Z,-Z),(-x_4e_\alpha,x_1e_\alpha)],(-x_4e_\beta,x_1e_\beta))^2 \\
&+\unc\sum_{\substack{\alpha,\beta}}  \tfrac{1}{x_1^2} g([(e_\alpha,0),(e_\beta,0)]_{\qg_3},(Z,-Z))^2 \\
&+\unc\sum_{\substack{\alpha,\beta}}  \tfrac{1}{x_1^2(x_1x_2-x_4^2)^2} g([(-x_4e_\alpha,x_1e_\alpha),(-x_4e_\beta,x_1e_\beta)]_{\qg_3},(Z,-Z))^2 \\
&+\unm\sum_{\substack{\alpha,\beta}} \tfrac{1}{x_1^2(x_1x_2-x_4^2)}g([(e_\alpha,0),(-x_4e_\beta,x_1e_\beta)]_{\qg_3},(Z,-Z))^2 
-\kil_\hg(Z,Z).   
\end{align*}
Using that $[\qg,\qg]\subset\kg$ and that $(Z,W)_{\qg_3}=(\tfrac{Z-W}{2}, -\tfrac{Z-W}{2})$ for any $Z,W\in\kg$, we obtain that  
\begin{align*}
\ricci(g)(\overline{Z},\overline{Z})  
=& -\unm\sum_{\substack{\alpha,\beta}} \tfrac{1}{x_1^2} g(([Z,e_\alpha],0),(e_\beta,0))^2 \\
& -\unm\sum_{\substack{\alpha,\beta}} \tfrac{1}{x_1^2(x_1x_2-x_4^2)} g((-x_4[Z,e_\alpha],-x_1[Z,e_\alpha]),(e_\beta,0))^2 \\
& -\unm\sum_{\substack{\alpha,\beta}} \tfrac{1}{x_1^2(x_1x_2-x_4^2)^2} g((-x_4[Z,e_\alpha],-x_1[Z,e_\alpha]),(-x_4e_\beta,x_1e_\beta))^2 \\
&+\unc\sum_{\substack{\alpha,\beta}}  \tfrac{1}{x_1^2} g(\unm([e_\alpha,e_\beta],-[e_\alpha,e_\beta]),(Z,-Z))^2 \\
&+\unc\sum_{\substack{\alpha,\beta}}  \tfrac{1}{x_1^2(x_1x_2-x_4^2)^2} g(\tfrac{x_4^2-x_1^2}{2}([e_\alpha,e_\beta],-[e_\alpha,e_\beta]),(Z,-Z))^2 \\
&+\unm\sum_{\substack{\alpha,\beta}} \tfrac{1}{x_1^2(x_1x_2-x_4^2)}g(-\tfrac{x_4}{2}([e_\alpha,e_\beta],-[e_\alpha,e_\beta]),(Z,-Z))^2 -\kil_\hg(Z,Z), 
\end{align*}
which gives the stated formula:
\begin{align*}
& \ricci(g)(\overline{Z},\overline{Z})  \\ 
=& -\unm\sum_{\substack{\alpha,\beta}} \tfrac{x_1^2}{x_1^2} (-\kil_\hg)([Z,e_\alpha],e_\beta)^2 
 -\unm\sum_{\substack{\alpha,\beta}} \tfrac{(x_1(-x_4)+x_4(-x_1))^2}{x_1^2(x_1x_2-x_4^2)} (-\kil_\hg)([Z,e_\alpha],e_\beta)^2 \\
& -\unm\sum_{\substack{\alpha,\beta}} \tfrac{\left(x_1x_4^2+x_2(-x_1^2)+x_4(-x_4x_1+x_1x_4)\right)^2}{x_1^2(x_1x_2-x_4^2)^2} (-\kil_\hg)([Z,e_\alpha],e_\beta)^2 
+\unc\sum_{\substack{\alpha,\beta}}  \tfrac{x_3^2}{x_1^2} (-\kil_\hg)([e_\alpha,e_\beta],Z)^2 \\
& +\unc\sum_{\substack{\alpha,\beta}}  \tfrac{x_3^2(x_4^2-x_1^2)^2}{x_1^2(x_1x_2-x_4^2)^2} (-\kil_\hg)([e_\alpha,e_\beta],Z)^2
+\unm\sum_{\substack{\alpha,\beta}} \tfrac{x_3^2x_4^2}{x_1^2(x_1x_2-x_4^2)}(-\kil_\hg)([e_\alpha,e_\beta],Z)^2 -\kil_\hg(Z,Z) \\ 
=& \left((a_l-1)\left(-1
-\tfrac{2x_4^2}{x_1x_2-x_4^2}
+ \tfrac{x_3^2}{4x_1^2} 
+\tfrac{x_3^2(x_1^2-x_4^2)^2}{4x_1^2(x_1x_2-x_4^2)^2} 
+\tfrac{x_3^2x_4^2}{2x_1^2(x_1x_2-x_4^2)}\right) - 1\right)\kil_\hg(Z,Z) \\ 
=& \left(-a_l+(a_l-1)\left(
-\tfrac{2x_4^2}{x_1x_2-x_4^2}
+ \tfrac{x_3^2}{4x_1^2} 
+\tfrac{x_3^2(x_1^2-x_4^2)^2}{4x_1^2(x_1x_2-x_4^2)^2} 
+\tfrac{x_3^2x_4^2}{2x_1^2(x_1x_2-x_4^2)}\right) \right)\kil_\hg(Z,Z).    
\end{align*}

Finally, it is easy to check by using \eqref{LB} and \eqref{Rc} that 
$$
\ricci(g)(\overline{X},\overline{Z}) = \ricci(g)(\overline{Y},\overline{Z}) =0, \qquad\forall \overline{X}\in\pg_1, \quad \overline{Y}\in\pg_2, \quad \overline{Z}\in\pg_3,
$$
concluding the proof of the proposition.  
\end{proof}

\begin{corollary}\label{gE}
A metric $g=(x_1,x_2,1,x_4)$, $x_4\ne 0$, is Einstein, say $\ricci(g)=\rho g$, if and only if $r_1=r_2=r_{12} =r_{3,0}=r_{3,1}=\dots=r_{3,t}=\rho$, where
\begin{align*}
r_1 
=& - \frac{x_2}{8x_1(x_1x_2-x_4^2)} 
+\frac{x_4^2}{2(x_1x_2-x_4^2)} + \frac{1}{2x_1},\\ 
r_2 
=& -\frac{x_1}{8x_2(x_1x_2-x_4^2)} 
+\frac{x_4^2}{2(x_1x_2-x_4^2)} + \frac{1}{2x_2}, \\ 
r_{12}  
=& \frac{4x_1x_2-1}{8(x_1x_2-x_4^2)}, \qquad
r_{3,l}= \frac{1}{2}(a_l(1-R)+R) \quad (\mbox{see \eqref{defR}}).
\end{align*}
\end{corollary}

\begin{proof}
Using Proposition \ref{ri} and the following equalities from \eqref{defg4},
$$
\begin{array}{c}
g(\overline{X},\overline{X}) = -x_1\kil_\hg(X,X), \quad   
g(\overline{Y},\overline{Y}) =-x_2\kil_\hg(Y,Y), \quad  
g(\overline{X},\overline{Y}) = -x_4\kil_\hg(X,Y)\\ 
g(\overline{Z},\overline{Z}) = -2x_3\kil_\hg(Z,Z), \quad 
g(\overline{X},\overline{Z}) = 0, \quad 
g(\overline{Y},\overline{Z}) = 0,
\end{array}
$$
we obtain that $\ricci(g)=\rho g$ if and only if
\begin{align*}
-x_1\rho =& \tfrac{1}{8x_1} + \tfrac{x_4^2}{8x_1(x_1x_2-x_4^2)} 
-\tfrac{x_1x_4^2}{2(x_1x_2-x_4^2)} - \unm,\\ 
-x_2\rho =& \tfrac{x_1}{8(x_1x_2-x_4^2)} 
+\tfrac{x_1x_2-x_4^2}{8x_1}
-\tfrac{(x_1x_2+x_4^2)^2}{8x_1(x_1x_2-x_4^2)} - \unm, \\  
-x_4\rho =& \tfrac{x_4}{8(x_1x_2-x_4^2)} 
-\tfrac{x_4}{4}
- \tfrac{x_4(x_1x_2+x_4^2)}{4(x_1x_2-x_4^2)}, \\ 
-2\rho
=& -a_l + (1-a_l)\left(
\tfrac{2x_4^2}{x_1x_2-x_4^2}
- \tfrac{1}{4x_1^2} 
-\tfrac{(x_1^2-x_4^2)^2}{4x_1^2(x_1x_2-x_4^2)^2} 
-\tfrac{x_4^2}{2x_1^2(x_1x_2-x_4^2)}\right), 
\end{align*}
from which the corollary easily follows.  
\end{proof}

The non-diagonal setting provides many Einstein metrics different from those given by Theorem \ref{HHK-E}, including one which exists on any $M=H\times H/\Delta K$ with $H/K$ an irreducible symmetric space.    

\begin{theorem}\label{End}
Let $H/K$ be an irreducible symmetric space.  
\begin{enumerate}[{\rm (i)}]
\item If $\kil_\kg=a\kil_\hg|_\kg$, $a>0$ (see Table \ref{table1}), then the following are Einstein metrics on the homogeneous space $M=H\times H/\Delta K$:
\begin{enumerate}[{\small $\bullet$}]
\item $g_1=(x_+,x_+,1,0)$ and $g_2=(x_-,x_-,1,0)$ if $a<\unm$ (see Theorem \ref{HHK-E} and Remark \ref{E2z-rem2}), where 
$$
x_\pm=\frac{1\pm \sqrt{1 - a(3-2a)}}{2a}.  
$$ 
\item $g_3=(1,1,1,y_+)$ and $g_4=(1,1,1,y_-)$ if $a>\unm$, where 
$$
y_\pm=\pm\frac{1}{2}\sqrt{\frac{2a-1}{2-a}}.  
$$
\item $g_5=\left(\unm,\frac{3}{2},1,\unm\right)$ and $g_6=\left(\frac{3}{2},\unm,1,\unm\right)$.  
\end{enumerate}
Moreover, in the case when $K$ is simple, these are all the $H\times H$-invariant Einstein metrics (up to scaling) on $M=H\times H/\Delta K$.   

\item The normalized scalar curvatures $\scalar_N(g_i):=\scalar(g_i)(\det_{\gk}{g_i})^{\frac{1}{\dim{M}}}$ are given as follows:
$$
\scalar_N(g_1)=\frac{(2n+d)(4x_+-1)}{8(x_+)^{2\alpha}}, \qquad 
\scalar_N(g_2)=\frac{(2n+d)(4x_--1)}{8(x_-)^{2\alpha}}, 
$$
$$
\scalar_N(g_3)=\scalar_N(g_4)=\frac{3(2n+d)}{8} \left(\frac{4(2-a)}{3(3-2a)}\right)^\alpha, \qquad 
\scalar_N(g_5)=\scalar_N(g_6)=(2n+d)2^{\alpha-2}, 
$$
where $\alpha=\tfrac{n+d}{2n+d}$.   

\item The metrics $g_5$ and $g_6$ are also Einstein on $M=H\times H/\Delta K$ for any $H/K$.  
\end{enumerate}
\end{theorem}

\begin{remark}
The only irreducible symmetric space with $a=0$ is $\SU(2)/S^1$, giving rise to $M=\SU(2)\times\SU(2)/\Delta S^1$.  A complete classification of $G$-invariant Einstein metrics on $M$ was obtained in \cite{AlkDttFrr}, see also  \cite[Example 6.10]{BhmWngZll} for further information. 
\end{remark} 

\begin{remark}\label{iss}
In addition to those in Table \ref{table1} satisfying that $\kil_\kg=a\kil_\hg|_\kg$, $a>0$ as in part (i), the remaining irreducible symmetric spaces $H/K$ are given by (see \cite[Table 7.102]{Bss}): 
\begin{enumerate}[{\small $\bullet$}]
\item $\SU(p+q)/\SU(p)\times\SU(q)\times S^1$, \quad $2\leq p\leq q$,   
\item $\SO(2m)/\U(m)$, \quad $m\geq 3$,  
\item $\SO(p+q)/\SO(p)\times\SO(q)$, \quad $2\leq p<q$, \quad $p+q\geq 7$,  
\item $\Spe(m)/\SU(m)\times S^1$, \quad $m\geq 2$, 
\item $\Spe(p+q)/\Spe(p)\times\Spe(q)$, \quad $1\leq p<q$,  
\item $G_2/\SO(4)$, \quad 
$F_4/\Spe(3)\times\SU(2)$, \quad
 $E_6/\SU(6)\times\SU(2)$, \quad
 $E_6/\SO(10)\times S^1$, 
\item $E_7/\SO(12)\times\SU(2)$, \quad
 $E_7/E_6\times S^1$, \quad
 $E_8/E_7\times\SU(2)$.  
\end{enumerate}
Remarkably, according to part (iii) of the above theorem, $g_5$ and $g_6$ are also Einstein on each of the corresponding homogeneous spaces $M=H\times H/\Delta K$.  On the other hand, there exist Einstein metrics on some of these spaces which are given as different scalings of $\gk$ on each $\pg_3^l$, $l=0,1,\dots,t$ (e.g., for $H/K=\Spe(m)/\U(m)$, see \cite{Rical}).  
\end{remark} 

\begin{remark}
The automorphism of $G=H\times H$ interchanging the two copies of $H$ defines an isometry between $g_5$ and $g_6$.  On the other hand, $g_3$ and $g_4$ are isometric via $(id,\theta)\in\Aut(G/\Delta K)$, where $\theta$ is the idempotent automorphism of $H$ such that $d\theta|_\kg=I$ and $d\theta|_\qg=-I$.
\end{remark} 

\begin{proof}[Proof of Theorem \ref{End}]
We first assume that a metric of the form $g=(x_1,x_2,x_3,x_4)$ is Einstein.  It follows from Theorem \ref{HHK-E} and Remark \ref{E2z-rem2} that, up to scaling, the only diagonal (i.e., $x_4=0$) Einstein metrics on $M$ are $g_1$ and $g_2$.  We therefore also assume from now on that $x_3=1$ and $x_4\ne 0$ and use the Einstein equations given in Corollary \ref{gE}.  

A simple algebraic manipulation gives that $r_1=r_2$ if and only if 
$$
\tfrac{x_1^2-x_2^2}{8x_1x_2(x_1x_2-x_4^2)}  = \tfrac{x_1-x_2}{2x_1x_2}, 
$$
so either $x_1=x_2$ or $x_1+x_2 = 4(x_1x_2-x_4^2)$.  Furthermore, it is straightforward to see that $r_1=r_{12}$ is equivalent to 
$$
x_1 - x_2 = 4(x_1x_2-x_4^2)(x_1-1);  
$$
in particular, either $x_1=x_2=1$ or $x_1x_2-x_4^2=\unm$ (i.e., $x_1+x_2=2$).    

In the case when $x_1=x_2=1$, we obtain that $R\ne 1$, so $\kil_\kg=a\kil_\hg|_\kg$ by the formulas for $r_{3,l}$ and $r_{12}=r_3:=r_{3,0}=\dots=r_{3,t}$ if and only if 
$$
\tfrac{3}{8(1-x_4^2)} 
= \tfrac{a}{2} + (1-a)\left(
\tfrac{1}{4} 
-\tfrac{3x_4^2}{4(1-x_4^2)}\right), 
$$
which is equivalent to $x_4^2=\tfrac{2a-1}{4(2-a)}$, giving rise to the Einstein metrics $g_3$ and $g_4$.  

The other possibility is therefore $x_1+x_2=2$ and $x_1x_2-x_4^2=\unm$.  A long but straightforward algebraic manipulation gives that under these assumptions, $r_{12}=r_{3,l}$ holds if and only if 
$$
(4x_1^2-8x_1+3)(5-4a_l)=0,
$$
providing the Einstein metrics $g_4$ and $g_5$ independently of $H/K$ since $R=1$ for these metrics.  Thus parts (i) and (iii) follow.  Note that the last sentence in (i) follows from Remark \ref{all} and the fact that the isotropy representation $\qg$ of $H/K$ is of real type as an $\Ad(K)$-representation for all the irreducible symmetric spaces with $\kil_\kg=a\kil_\hg|_\kg$, $a>0$, since the only other possibility is to be of complex type as for, precisely, hermitian symmetric spaces (see \cite{Wlf} or \cite{WngZll1}).  


Concerning part (ii), recall that the values for $\scalar_N(g_1)$ and $\scalar_N(g_2)$ are given in Theorem \ref{HHK-E}.  For the rest of the Einstein metrics we proceed as follows.  If $g=(x_1,x_2,1,x_4)$ is Einstein, then $\det_{\gk}{g}=(x_1x_2-x_4^2)^n$ by \eqref{mat} and it follows from the formula for $r_{12}$ in Corollary \ref{gE} that 
$$
\scalar_N(g) = (2n+d)\tfrac{4x_1x_2-1}{8(x_1x_2-x_4^2)}(x_1x_2-x_4^2)^{\tfrac{n}{2n+d}}
= \tfrac{(2n+d)(4x_1x_2-1)}{8(x_1x_2-x_4^2)^\alpha},
$$
concluding the proof.  
\end{proof}

\begin{table}
{\small 
$$
\begin{array}{c|c|c|c|c}
H/K & m & \scalar_N(g_1) & \scalar_N(g_2) & \scalar_N(g_5)     
\\[2mm] \hline \hline \rule{0pt}{14pt}
\SU(m)/\SO(m) & \geq 3 & \eqref{scg1su} & \eqref{scg2su} &  2(3m^2 + m- 4) 4^{-\tfrac{2m+3}{3m+4}}  
\\[2mm]  \hline \rule{0pt}{14pt}
\SO(2m)/\SO(m)\times\SO(m) & \geq 4 & \eqref{scg1so} & \eqref{scg2so}   &  m(3m-1) 2^{-\frac{4m-1}{3m-1}}   
\\[2mm]  \hline \rule{0pt}{14pt}
G_2/\SU(2)\times\SU(2) & & 7.8598 & 8.0237 &  \frac{11}{2} 2^{\frac{7}{11}} \approx 8.5492  
\\[2mm]  \hline \rule{0pt}{14pt} 
E_6/\Spe(4) &  & 44.0481 & 44.3085 & 30\, 2^{\frac{13}{20}} \approx 47.0750  
\\[2mm]  \hline \rule{0pt}{14pt} 
E_7/\SU(8) &  & 75.0853 & 75.3101 &  \frac{203}{4} 2^{\frac{19}{29}}\approx 79.9213  
\\[2mm]  \hline \rule{0pt}{14pt} 
E_8/\SO(16) &  & 139.8741 & 140.0578 &  94\, 2^{\frac{31}{47}} \approx 148.4839  
\\[2mm] \hline\hline
\end{array}
$$}
\caption{Normalized scalar curvature of Einstein metrics in Theorem \ref{End} such that $a<\unm$ (see Table \ref{table1}).} \label{Sc1}
\end{table}

\begin{table}
{\small 
$$
\begin{array}{c|c|c|c}
H/K & m & \scalar_N(g_3) & \scalar_N(g_5)     
\\[2mm] \hline \hline \rule{0pt}{14pt}
\SU(2m)/\Spe(m) & \geq 2 &  \frac{1}{8}  (6m^2-m-2) \,3^{\frac{m-1}{3m-2}} \left(  \frac{2(3m-1)}{ 2m-1} \right)^{\frac{2m-1}{3m-2}}   &  (6m^2-m-2) 2^{-\frac{4m-3}{3m-2}}  
\\[2mm]  \hline \rule{0pt}{14pt}
\SO(m)/\SO(m-1) & \geq 7 & \frac{1}{16}  (m^2+m-2) \,9^{\frac{1}{m+2}} \left(  \frac{4(m-1)}{ m} \right)^{\frac{m}{m+2}}  &  (m^2+m-2) 4^{-\frac{m+3}{m+2}}  
\\[2mm]  \hline \rule{0pt}{14pt}
\Spe(2m)/\Spe(m)\times\Spe(m) & \geq 1 & (6m^2+m) (\frac{9}{16})^{\tfrac{m}{6m+1}}  \left(  \tfrac{3m+1}{4m+1} \right)^{\tfrac{4m+1}{6m+1}} &  (6m^2+m) 4^{-\frac{m}{6m+1}} 
\\[2mm]  \hline \rule{0pt}{14pt}
F_4/\SO(9) &  & \frac{17}{26} (44)^{\frac{13}{17}}  (39)^{\frac{4}{17}}   \approx 27.9641    & 17\, 2^{\frac{13}{17}} \approx 28.8834 
\\[2mm]  \hline \rule{0pt}{14pt} 
E_6/F_4 &  &  \frac{13}{\sqrt{3} } 10^{\frac{3}{4}}   \approx 42.2068 &  26\, 2^{\frac{3}{4}} \approx 43.7266 
\\[2mm] \hline\hline
\end{array}
$$}
\caption{Normalized scalar curvature of Einstein metrics in Theorem \ref{End} such that $a>\unm$ (see Table \ref{table1}).} \label{Sc2}
\end{table}

The explicit values for the normalized scalar curvatures of the Einstein metrics given in Theorem \ref{End} are given in Tables \ref{Sc1} and \ref{Sc2}, including the following formulas:
\begin{align}
\SU(m)/\SO(m), \quad \scalar(g_1)=\tfrac{  (3m^2+m-4) (3m+2 +4 \sqrt{m+2}) \left(  \tfrac{m-2}{ 1+\sqrt{m+2}}  \right)^{\tfrac{4(m+1)}{3m+4}} }{16 m -32},  \label{scg1su}\\ 
\SU(m)/\SO(m), \quad\scalar(g_2)=\tfrac{  (3m^2+m-4) (3m+2 -4 \sqrt{m+2})  \left(  \tfrac{m-2}{ 1-\sqrt{m+2}} \right)^{\tfrac{4(m+1)}{3m+4}} }{16 m -32}, \label{scg2su}\\
\SO(2m)/\SO(m)^2, \quad\scalar(g_1)=\tfrac{  m(3m-1) (3m-1+2\sqrt{2m}) (\left(  \tfrac{2m-2+ \sqrt{2m}}{2(m-2)}  \right)^{\tfrac{-4m+2}{3m-1}} }{8( m -2)},  \label{scg1so} \\
\SO(2m)/\SO(m)^2, \quad\scalar(g_2)=\tfrac{  m(3m-1) (3m-1-2\sqrt{2m}) (\left(  \tfrac{2m-2- \sqrt{2m}}{2(m-2)}  \right)^{\tfrac{-4m+2}{3m-1}} }{8( m -2)}.  \label{scg2so}
\end{align}

It is straightforward to see that these normalized scalar curvatures satisfy the following inequalities.  

\begin{proposition}\label{nonhom}
$\scalar_N(g_1)<\scalar_N(g_2)<\scalar_N(g_5)$ for any $a<\unm$ and $\scalar_N(g_3)<\scalar_N(g_5)$ for any $a>\unm$.  In particular, all these metrics are pairwise non-homothetic. 
\end{proposition}

\appendix 
\section{Tables}\label{table-sec}

We list in this appendix all the compact homogeneous spaces $H/K$ such that the isotropy representation Casimir operator (see \eqref{cas}) is given by $\cas_\chi=\kappa I_\qg$ for some $\kappa\in\RR$ and the Killing forms satisfy that $\kil_\kg=a\kil_\hg|_\kg$ for some $a\in\RR$.  This was obtained using results from \cite{DtrZll} and the tables given in \cite{stab} and \cite{Sch}.  
  
\begin{table}
{\small 
$$
\begin{array}{c|c|c|c|c|c|c}
H/K & m & \dim{H} & d=\dim{K} & n=\dim{H/K} & a & \eqref{cond2}    
\\[2mm] \hline \hline \rule{0pt}{14pt}
\SU(m)/\SO(m) & \geq 2 & m^2-1 & \tfrac{m(m-1)}{2} & \tfrac{(m+2)(m-1)}{2}  & \tfrac{m-2}{2m}  &  \checkmark 
\\[2mm]  \hline \rule{0pt}{14pt}
\SU(2m)/\Spe(m) & \geq 2 & 4m^2-1 & m(2m+1) & (m-1)(2m+1)  & \frac{m+1}{2m}  & {\rm No}
\\[2mm]  \hline \rule{0pt}{14pt}
\SO(m)/\SO(m-1) & \geq 7 & \tfrac{m(m-1)}{2} & \tfrac{(m-1)(m-2)}{2} & m-1  &  \tfrac{m-3}{m-2}  & {\rm No}
\\[2mm]  \hline \rule{0pt}{14pt}
\SO(2m)/\SO(m)\times\SO(m) & \geq 4 & m(2m-1)& m(m-1) & m^2 & \tfrac{m-2}{2(m-1)}  &\checkmark 
\\[2mm]  \hline \rule{0pt}{14pt}
\Spe(2m)/\Spe(m)\times\Spe(m) & \geq 1 & 2m(4m+1) & 2m(2m+1) & 4m^2 &  \tfrac{m+1}{2m+1}  & {\rm No}
\\[2mm]  \hline \rule{0pt}{14pt}
F_4/\SO(9) &  & 52 & 36 & 16 & \tfrac{7}{9} & {\rm No}
\\[2mm]  \hline \rule{0pt}{14pt} 
E_6/\Spe(4) &  & 78 & 36 & 42 & \tfrac{5}{12} & \checkmark 
\\[2mm]  \hline \rule{0pt}{14pt} 
E_6/F_4 &  & 78 & 52 & 26 & \tfrac{3}{4} & {\rm No}
\\[2mm]  \hline \rule{0pt}{14pt} 
E_7/\SU(8) &  & 133 & 63 & 70 & \tfrac{4}{9} & \checkmark 
\\[2mm]  \hline \rule{0pt}{14pt} 
E_8/\SO(16) &  & 248 & 120 & 128 & \tfrac{7}{15} & \checkmark  
\\[2mm] \hline\hline
\end{array}
$$}
\caption{Irreducible symmetric spaces $H/K$ such that $\kil_\kg=a\kil_\hg|_\kg$ for some $a\in\RR$  (e.g.\ $K$ simple, see \cite[7.102]{Bss}).  In this case, $\kappa=\rho=\unm$, $a=\tfrac{2d-n}{2d}$, and condition \eqref{cond2} is equivalent to $d< n$ (or $a<\unm$).} \label{table1}
\end{table}

\begin{table}
{\small
$$
\begin{array}{c|c|c|c|c}
H/K & m & \dim{H} & d=\dim{K} & n=\dim{H/K}    
\\[2mm] \hline \hline \rule{0pt}{14pt}
\SO(d)/K &  & \tfrac{d(d-1)}{2} & d>3 & \tfrac{d(d-3)}{2}  
\\[2mm]  \hline \rule{0pt}{14pt}
\SU(\tfrac{m(m-1)}{2})/\SU(m) & \geq 5 & \tfrac{m^2(m-1)^2-4}{4} & m^2-1 &   \tfrac{m^2(m^2-2m-3)}{4}
\\[2mm]  \hline \rule{0pt}{14pt}
\SU(\tfrac{m(m+1)}{2})/\SU(m) & \geq 3 & \tfrac{m^2(m+1)^2-4}{4} & m^2-1 &   \tfrac{m^2(m^2+2m-3)}{4} 
\\[2mm]  \hline \rule{0pt}{14pt}
\SU(m^2)/\SU(m)^2& \geq 3 & m^4-1 & 2(m^2-1) & (m^2-1)^2 
\\[2mm]  \hline \rule{0pt}{14pt}
\SO((m-1)(2m+1))/\Spe(m)& \geq 3 & \tfrac{(m-1)(2m+1)(2m^2-m-2)}{2}   & m (2m+1) &  \tfrac{(2m+1)(2m^3-3m^2-3m+2)}{2}
\\[2mm]  \hline \rule{0pt}{14pt}
\SO(\tfrac{(m-1)(m+2)}{2})/\SO(m)& \geq 5 & \tfrac{(m+2) (m-1)(m^4+m-4)}{8}    & \tfrac{m(m-1)}{2} &  \tfrac{(m-1)(m^3+3m^2-6m-8)}{8}
\\[2mm] \hline\hline
\end{array}
$$}
\caption{Isotropy irreducible $H/K$ (non-symmetric) such that $\kil_\kg=a\kil_\hg|_\kg$ for some $a\in\RR$  (e.g.\ $K$ simple): {\bf families} (see \cite[7.106, 7.107]{Bss} or \cite[Table 1]{Sch}).  See Table \ref{table3-fam2} for the values of $a,\kappa,\rho$ for these spaces.  In the first line, $K$ is simple and the adjoint representation embedding is considered.} \label{table3-fam1}
\end{table}

\begin{table}
{\small
$$
\begin{array}{c|c|c|c|c}
H/K & a & \kappa & \rho & \eqref{cond2}   
\\[2mm] \hline \hline \rule{0pt}{14pt}
\SO(d)/K & \tfrac{1}{d-2} & \tfrac{2}{d-2} & \frac{d+2}{4(d-2)} &  \checkmark 
\\[2mm]  \hline \rule{0pt}{14pt}
\SU(\tfrac{m(m-1)}{2})/\SU(m)  & \frac{2}{(m-1)(m-2)}  & \frac{4}{m(m-2)} & \unc+\frac{2}{m(m-2)} &  \checkmark
\\[2mm]  \hline \rule{0pt}{14pt}
\SU(\tfrac{m(m+1)}{2})/\SU(m)  & \frac{2}{(m+1)(m+2)} & \frac{4}{m(m+2)} & \unc+\frac{2}{m(m+2)} & \checkmark
\\[2mm]  \hline \rule{0pt}{14pt}
\SU(m^2)/\SU(m)^2 & \frac{1}{m^2} & \frac{2}{m^2} & \unc+\frac{1}{m^2}  & \checkmark 
\\[2mm]  \hline \rule{0pt}{14pt}
\SO((m-1)(2m+1))/\Spe(m) & 1-  \tfrac{2m^3-3m^2-3m+2}{m (m^2-1) (2m-3)} & \frac{2}{(m^2-1)(2m-3)} & \unc+\frac{1}{(m^2-1)(2m-3)}  & (*) 
\\[2mm]  \hline \rule{0pt}{14pt}
\SO(\tfrac{(m-1)(m+2)}{2})/\SO(m) &  \frac{2}{(m+3)(m+2)}  & \frac{4m}{(m^2-4)(m+3)} & \unc+\frac{2m}{(m^2-4)(m+3)}  & \checkmark
\\[2mm] \hline\hline
\end{array}
$$}
\caption{Isotropy irreducible $H/K$ (non-symmetric) such that $\kil_\kg=a\kil_\hg|_\kg$ for some $a\in\RR$  (e.g.\ $K$ simple): {\bf families} (see \cite[7.106, 7.107]{Bss} or \cite[Table 1]{Sch}).  Here $\kappa=\tfrac{d(1-a)}{n}=2\rho-\unm$ and $a=\tfrac{d-n\kappa}{d}$ (see Table \ref{table3-fam1} for further information on these spaces).  (*) ${\rm No}:\, 3,4,5,6; \quad \checkmark:  \,m \ge 7$} \label{table3-fam2}
\end{table}

\begin{table}
{\small
$$
\begin{array}{c|c|c|c|c|c|c|c}
H/K & \dim{H} & d & n & a & \kappa & \rho & \eqref{cond2}   
\\[2mm] \hline \hline \rule{0pt}{14pt}
\SU(16)/\SO(10) & 255 & 45 & 210 &  \frac{1}{8} &  \frac{3}{16}  & \frac{11}{32} & \checkmark    
\\[2mm]  \hline \rule{0pt}{14pt}
\SU(27)/E_6 & 728 & 78 & 650 &  \frac{2}{27} &  \frac{1}{9}  & \frac{11}{36} &  \checkmark 
\\[2mm]  \hline \rule{0pt}{14pt}
\SO(7)/G_2 & 21 & 14 & 7 &  \frac{4}{5}   &  \frac{2}{5}  & \frac{9}{20} & {\rm No}   
\\[2mm]  \hline \rule{0pt}{14pt}
\Spe(2)/\SU(2) & 10 & 3 & 7 &  \frac{1}{15} &  \frac{2}{5}  & \frac{9}{20} & \checkmark
\\[2mm]  \hline \rule{0pt}{14pt}
\Spe(7)/\Spe(3) & 105  & 21  & 84  & \frac{1}{10} &  \frac{9}{40} & \frac{29}{80} & \checkmark 
\\[2mm]  \hline \rule{0pt}{14pt}
\Spe(10)/\SU(6) & 210 & 35 & 175 &  \frac{1}{11}  &  \frac{2}{11}  & \frac{15}{44} & \checkmark 
\\[2mm]  \hline \rule{0pt}{14pt}
\Spe(16)/\SO(12) & 528 & 66 & 462 &  \frac{5}{68} &  \frac{9}{68} & \frac{43}{136} &  \checkmark 
\\[2mm]  \hline \rule{0pt}{14pt}
\Spe(28)/E_7 & 1596 & 133 & 1463 &  \frac{3}{58} &  \frac{5}{58}  & \frac{17}{58} & \checkmark 
\\[2mm]  \hline \rule{0pt}{14pt}
%
\SO(16)/\SO(9) & 120 & 36 & 84 &  \frac{1}{4} &  \frac{9}{28}  & \frac{23}{56} & \checkmark 
\\[2mm]  \hline \rule{0pt}{14pt}
\SO(26)/F_4 & 325 & 52 & 273 &  \frac{1}{8} &  \frac{1}{6}  & \frac{1}{3} & \checkmark 
\\[2mm]  \hline \rule{0pt}{14pt}
\SO(42)/\Spe(4) & 861 & 36 & 825 &  \frac{1}{56} &  \frac{3}{70}  & \frac{19}{70} &  \checkmark 
\\[2mm]  \hline \rule{0pt}{14pt}
\SO(70)/\SU(8) & 2415 & 63 & 2352 &  \frac{1}{85} &  \frac{9}{340}  & \frac{179}{680} &  \checkmark 
\\[2mm]  \hline \rule{0pt}{14pt}
\SO(128)/\SO(16) & 8128 & 120 & 8008 & \frac{1}{144} & \frac{5}{336}  & \frac{173}{672} &  \checkmark 
\\[2mm] \hline\hline
\end{array}
$$}
\caption{Isotropy irreducible $H/K$ (non-symmetric) such that $\kil_\kg=a\kil_\hg|_\kg$ for some $a\in\RR$  (e.g.\ $K$ simple): isolated, $G$ {\bf classical} (see \cite[7.106, 7.107]{Bss} or \cite[Table 3]{Sch}).  Here $\kappa=\tfrac{d(1-a)}{n}=2\rho-\unm$ and $a=\tfrac{d-n\kappa}{d}$.} \label{table3-clas}
\end{table}

\begin{table}
{\small
$$
\begin{array}{c|c|c|c|c|c|c|c}
H/K & \dim{H} & d & n & a & \kappa & \rho & \eqref{cond2}   
\\[2mm] \hline \hline \rule{0pt}{14pt}
E_6/\SU(3) & 78 & 8 & 70 & \frac{1}{36} & \frac{1}{9} & \frac{11}{36} &  \checkmark 
\\[2mm]  \hline \rule{0pt}{14pt}
E_6/\SU(3)^3 & 78 & 24 & 54 &  \frac{1}{4} &  \frac{1}{3}  & \frac{5}{12} &  \checkmark  
\\[2mm]  \hline \rule{0pt}{14pt}
E_6/G_2 & 78 & 14 & 64 & \frac{1}{9} & \frac{7}{36} & \frac{25}{72} &   \checkmark  
\\[2mm]  \hline \rule{0pt}{14pt}
E_7/\SU(3) & 133 & 8 & 125 &  \frac{1}{126} &  \frac{4}{63}  & \frac{71}{252} &  \checkmark 
\\[2mm]  \hline \rule{0pt}{14pt}
E_7/G_2\times\Spe(3) & 133 & 35 & 98 & \frac{2}{9} & \frac{5}{18}  & \frac{7}{18} &  \checkmark 
\\[2mm]  \hline \rule{0pt}{14pt}
E_8/\SU(9) & 248 & 80 & 168 & \frac{3}{10} & \frac{1}{3}  & \frac{5}{12} &   \checkmark 
\\[2mm]  \hline \rule{0pt}{14pt}
F_4/\SU(3)^2 & 52 & 16 & 36 & \frac{1}{4}  & \frac{1}{3}  & \frac{5}{12} &  \checkmark 
\\[2mm]  \hline \rule{0pt}{14pt}
G_2/\SU(2) & 14 & 3 & 11 & \frac{1}{56} & \frac{15}{56}  & \frac{43}{112} &  \checkmark 
\\[2mm]  \hline \rule{0pt}{14pt}
G_2/\SU(3) & 14 & 8 & 6 &  \frac{3}{4} &  \frac{1}{3}  & \frac{5}{12} &  {\rm No} 
\\[2mm] \hline\hline
\end{array}
$$}
\caption{Isotropy irreducible $H/K$ (non-symmetric) such that $\kil_\kg=a\kil_\hg|_\kg$ for some $a\in\RR$  (e.g.\ $K$ simple): isolated $G$ {\bf exceptional} (see \cite[7.106, 7.107]{Bss} or \cite[Table 4]{Sch}).  Here $\kappa=\tfrac{d(1-a)}{n}=2\rho-\unm$ and $a=\tfrac{d-n\kappa}{d}$.} \label{table3-exc}
\end{table}

\begin{table}
{\small 
$$
\begin{array}{c|c|c|c|c|c|c}
H/K & \dim{H} & d=\dim{K} & n=\dim{H/K}  & \kappa & \rho & \eqref{cond2}   
\\[2mm] \hline \hline \rule{0pt}{14pt}
\SU(m)/T^{m-1} & m^2-1 & m-1 & m(m-1)  &  \tfrac{1}{m} &  \tfrac{m+2}{4m} & \checkmark
\\[2mm]  \hline \rule{0pt}{14pt} 
\SO(2m)/T^m & m(2m-1) & m & 2m(m-1) &  \tfrac{1}{2(m-1)} &  \tfrac{m}{4(m-1)}  & \checkmark 
\\[2mm]  \hline \rule{0pt}{14pt} 
E_6/T^6 & 78 & 6 & 72  &  \tfrac{1}{12} &  \tfrac{7}{24} &  \checkmark
\\[2mm]  \hline \rule{0pt}{14pt} 
E_7/T^7 & 133 & 7 & 126 & \tfrac{1}{18} & \tfrac{5}{18}  &   \checkmark
\\[2mm]  \hline \rule{0pt}{14pt} 
E_8/T^8 & 248 & 8 & 240  &  \tfrac{1}{30} &  \tfrac{4}{15}  & \checkmark 
\\[2mm] \hline\hline
\end{array}
$$}
\caption{Non-isotropy irreducible $H/K$ such that $G$ is simple, $K$ is {\bf abelian} and $\gk$ is Einstein (see \cite[Tables 1,2,3,4,9]{stab}).  Here $\kappa=\tfrac{d}{n}=2\rho-\unm$ and $\kil_\kg=0$ (i.e., $a=0$).} \label{table2-abel}
\end{table}

\begin{table}
{\small 
$$
\begin{array}{c|c|c|c|c|c|c|c}
H/K & \dim{H} & d=\dim{K} & n=\dim{H/K} & a & \kappa & \rho & \eqref{cond2}   
\\[2mm] \hline \hline \rule{0pt}{14pt}
\SO(8)/G_2 & 28 & 14 & 14 & \tfrac{2}{3} &  \tfrac{1}{3} &  \tfrac{5}{12} & {\rm No} 
\\[2mm]  \hline \rule{0pt}{14pt} 
F_4/\Spin(8) & 52 & 28 & 24 & \tfrac{2}{3} &  \tfrac{7}{18} &  \tfrac{4}{9}  & {\rm No} 
\\[2mm]  \hline \rule{0pt}{14pt} 
E_7/\SO(8) & 133 & 28 & 105 & \tfrac{1}{6} &  \tfrac{2}{9} &  \tfrac{13}{36} &  \checkmark
\\[2mm]  \hline \rule{0pt}{14pt} 
E_8/\SO(5) & 248 & 10 & 238 & \tfrac{1}{120} & \tfrac{1}{24} & \tfrac{13}{48}  &   \checkmark
\\[2mm]  \hline \rule{0pt}{14pt} 
E_8/\SO(9) & 248 & 36 & 212 & \tfrac{7}{60} &  \tfrac{3}{20} &  \tfrac{13}{40}  & \checkmark 
\\[2mm]  \hline \rule{0pt}{14pt} 
E_8/\Spin(9) & 248 & 36 & 212 & \tfrac{7}{60} &  \tfrac{3}{20} &  \tfrac{13}{40}  & \checkmark
\\[2mm] \hline\hline
\end{array}
$$}
\caption{Non-isotropy irreducible $H/K$ such that $G$ is simple, $K$ is {\bf simple} (in particular, $\kil_\kg=a\kil_\hg|_\kg$ for some $a>0$) and $\gk$ is Einstein (see \cite[Tables 1,2,3,4,9]{stab}).  Here $\kappa=\tfrac{d(1-a)}{n}=2\rho-\unm$ and $a=\tfrac{d-n\kappa}{d}$.} \label{table2-simple}
\end{table}

\begin{table}
{\small 
$$
\begin{array}{c|c|c|c|c}
H/K & k,m & \dim{H} & d=\dim{K} & n=\dim{H/K}    
\\[2mm] \hline \hline \rule{0pt}{14pt}
\Spe(mk)/\Spe(k)^m & k\geq 1, m\geq 3 & mk(2mk+1) & mk(2k+1) & 2k^2m(m-1) 
\\[2mm]  \hline \rule{0pt}{14pt} 
\SO(mk)/\SO(k)^m & k\geq 3, m\geq 3 & \tfrac{mk(mk-1)}{2} & \tfrac{mk(k-1)}{2} &\tfrac{mk^2(m-1)}{2}  
\\[2mm]  \hline \rule{0pt}{14pt} 
\SO(m^2)/\SO(m)^2 & m\geq 3 & \tfrac{m^2(m^2-1)}{2} & m(m-1) & \tfrac{m(m+2)(m-1)^2}{2}  
\\[2mm]  \hline \rule{0pt}{14pt} 
\SO(4m^2)/\Spe(m)^2 & m\geq 2 & 2m^2(4m^2-1) & 2m(2m+1) & 2m(m-1)(2m+1)^2 
\\[2mm]  \hline \rule{0pt}{14pt} 
E_6/\SO(3)^3 && 78 & 9 & 69 
\\[2mm]  \hline \rule{0pt}{14pt} 
E_6/\SU(2)\times\SO(6) && 78 & 18 & 60 
\\[2mm]  \hline \rule{0pt}{14pt} 
E_7/\SU(2)^7 && 133 & 21 & 112 
\\[2mm]  \hline \rule{0pt}{14pt} 
E_8/\SU(5)^2 && 248 & 48 & 200 
\\[2mm]  \hline \rule{0pt}{14pt} 
E_8/\SU(3)^4 && 248 & 32 & 216  
\\[2mm]  \hline \rule{0pt}{14pt} 
E_8/\SO(3)^4 && 248 & 12 & 236 
\\[2mm]  \hline \rule{0pt}{14pt} 
E_8/\SU(2)^8 && 248 & 24 & 224  
\\[2mm]  \hline \rule{0pt}{14pt} 
E_8/\SO(5)^2 && 248 & 20 & 228  
\\[2mm]  \hline \rule{0pt}{14pt} 
E_8/\SU(3)^2 && 248 & 16 & 232  
\\[2mm] \hline\hline
\end{array}
$$}
\caption{Non-isotropy irreducible $H/K$ such that $G$ is simple, $K$ is {\bf semisimple non-simple}, $\gk$ is Einstein and $\kil_\kg=a\kil_\hg|_\kg$ for some $a>0$ (see \cite[Tables 1,2,3,4,9]{stab}).} \label{table2-ss-2}
\end{table}

\begin{table}
{\small 
$$
\begin{array}{c|c|c|c|c}
H/K & a & \kappa & \rho & \eqref{cond2}   
\\[2mm] \hline \hline \rule{0pt}{14pt}
\Spe(mk)/\Spe(k)^m & \tfrac{k+1}{mk+1} &   \tfrac{2k+1}{2(mk+1)} &  \tfrac{(m+2)k+2}{4(mk+1)} &  (*)
\\[2mm]  \hline \rule{0pt}{14pt} 
\SO(mk)/\SO(k)^m  & \tfrac{k-2}{mk-2} &  \tfrac{k-1}{mk-2} &  \tfrac{(m+2)k-4}{4(mk-2)} &  \checkmark 
\\[2mm]  \hline \rule{0pt}{14pt} 
\SO(m^2)/\SO(m)^2 & \tfrac{m-2}{m(m^2-2)} & \tfrac{2(m-1)}{m(m^2-2)} &  \tfrac{m^3+2m-4}{4m(m^2-2)}  & \checkmark 
\\[2mm]  \hline \rule{0pt}{14pt} 
\SO(4m^2)/\Spe(m)^2 & \tfrac{m+1}{2m(2m^2-1)}  & \tfrac{2m+1}{2m(2m^2-1)} &  \tfrac{2m^3+m+1}{4m(2m^2-1)}  & \checkmark 
\\[2mm]  \hline \rule{0pt}{14pt} 
E_6/\SO(3)^3  & \tfrac{1}{24} & \tfrac{1}{8} &  \tfrac{5}{16} &  \checkmark
\\[2mm]  \hline \rule{0pt}{14pt} 
E_6/\SU(2)\times\SO(6)  & \tfrac{1}{6} & \tfrac{1}{4}  &  \tfrac{3}{8} &  \checkmark
\\[2mm]  \hline \rule{0pt}{14pt} 
E_7/\SU(2)^7  & \tfrac{1}{9} &   \tfrac{1}{6} &  \tfrac{1}{3} &  \checkmark
\\[2mm]  \hline \rule{0pt}{14pt} 
E_8/\SU(5)^2  & \tfrac{1}{6} & \tfrac{1}{5} & \tfrac{7}{20}  &   \checkmark
\\[2mm]  \hline \rule{0pt}{14pt} 
E_8/\SU(3)^4  & \tfrac{1}{10} &  \tfrac{2}{15} &  \tfrac{19}{60}  & \checkmark 
\\[2mm]  \hline \rule{0pt}{14pt} 
E_8/\SO(3)^4 & \tfrac{1}{60} & \tfrac{1}{20} &  \tfrac{11}{40}  & \checkmark
\\[2mm]  \hline \rule{0pt}{14pt} 
E_8/\SU(2)^8 & \tfrac{1}{15} &  \tfrac{1}{10} &  \tfrac{3}{10}  & \checkmark 
\\[2mm]  \hline \rule{0pt}{14pt} 
E_8/\SO(5)^2  & \tfrac{1}{20} & \tfrac{1}{12} &  \tfrac{7}{24}  & \checkmark 
\\[2mm]  \hline \rule{0pt}{14pt} 
E_8/\SU(3)^2  & \tfrac{1}{30} &  \tfrac{1}{15} &  \tfrac{17}{60}  & \checkmark 
\\[2mm] \hline\hline
\end{array}
$$}
\caption{Non-isotropy irreducible $H/K$ such that $G$ is simple, $K$ is {\bf semisimple non-simple}, $\gk$ is Einstein and $\kil_\kg=a\kil_\hg|_\kg$ for some $a>0$ (see \cite[Tables 1,2,3,4,9]{stab}).  Here $\kappa=\tfrac{d(1-a)}{n}=2\rho-\unm$ and $a=\tfrac{d-n\kappa}{d}$.  (*): \eqref{cond2} holds if and only if $m\geq 5$ and $k\geq 2$; $m=4$ and $k\geq 4$; $m=3$ and $k\geq 9$; and equality holds in \eqref{cond2} if and only if $(m,k)=(3,8), (4,3)$.} \label{table2-ss}
\end{table}

\end{document}